\documentclass[11pt]{amsart}

\usepackage[authoryear]{natbib}

\RequirePackage[OT1]{fontenc}
\RequirePackage{amsthm,amsmath,amssymb}
\usepackage[colorlinks=true, urlcolor=blue, linkcolor=black, citecolor=black, pdftex]{hyperref}
\usepackage{graphicx}
\usepackage[usenames]{color}

\usepackage{epstopdf} 
\usepackage{memhfixc}  

\RequirePackage{amssymb}
 
\def\Var{{\rm Var}}

\DeclareMathAlphabet{\mathonebb}{U}{bbold}{m}{n}
\newcommand{\1}{\ensuremath{\mathonebb{1}}}
\def\1{1\hskip-2.6pt{\rm l}}
\def\N{{\mathbb{N}}}

\def\Z{{\mathbb{Z}}}
\def\M{{\mathbb{M}}}
 
\def\R{{\mathbb{R}}}
\def\C{{\mathbb{C}}}
\def\E{{\mathbb{E}}}
\def\P{{\mathbb{P}}}
\def\IL{{\mathbb{L}}}
\def\D{{\mathcal D}}

\def\m{{{\rm m}}}
\def\VV{{\mathcal V}}
\def\m{{\mathfrak{m}}}
\def\SS{{\mathbb{S}}}
\def\gT{{\mathbf{T}}}

\def\B{{{\mathcal B}}}

\def\F{{\mathcal{F}}}
\def\X{{\mathcal{X}}}
\def\T{{\mathcal T}}

\def\M{\mathcal M}
\def\PP{\mathcal P}

\def\EE{{\mathcal E}}
\def\LL{{\mathcal L}}
\def\RR{{\mathcal R}}
\def\S{{\mathcal S}}
\def\A{{\mathcal A}}
\def\CC{\mathcal C}
\def\crit{{\rm crit}}

\def\eps{{\varepsilon}}

\def\<{{\langle}}
\def\>{{\rangle}}

\newcommand{\pen}{\mathop{\rm pen}\nolimits}
\newcommand{\eref}[1]{(\ref{#1})}
\newcommand{\pa}[1]{\left({#1}\right)}
\newcommand{\norm}[1]{\left\|{#1}\right\|}
\newcommand{\cro}[1]{\left[{#1}\right]}
\newcommand{\ab}[1]{\left|{#1}\right|}
\newcommand{\ac}[1]{\left\{{#1}\right\}}

\newtheorem{thm}{Theorem}
\newtheorem{lemma}{Lemma}
\newtheorem{prop}{Proposition}
\newtheorem{cor}{Corollary}
\newtheorem{defi}{Definition}
\newtheorem{ass}{Assumption}
\newtheorem{ex}{Example}
\newtheorem{pb}{Problem}
\newtheorem{proc}{Procedure}

\def\A{{\mathcal A}}
\def\PP{\mathcal P}
\def\EE{{\mathcal E}}
\def\S{{\mathcal S}}

\def\M{\mathcal M}
\def\CC{\mathcal C}

\def\telque{\big |}

\begin{document}
\title[Estimator selection]{Estimator selection with respect to Hellinger-type risks}
\date{May, 9 2009} 
\author{Yannick Baraud} 
\address{Universit\'e de Nice Sophia-Antipolis, Laboratoire J-A Dieudonn\'e,
  Parc Valrose, 06108 Nice cedex 02} 
\email{baraud@unice.fr}
\begin{abstract}
We observe a random measure $N$ and aim at estimating its intensity $s$. This statistical framework allows to deal simultaneously with the problems of estimating a density, the marginals of a multivariate distribution, the mean of a random vector with nonnegative components and the intensity of a Poisson process. Our estimation strategy is based on estimator selection. Given a family of estimators of $s$ based on the observation of $N$, we propose a selection rule, based on $N$ as well, in view of selecting among these. Little assumption is made on the collection of estimators. The procedure offers the possibility to perform model selection and also to select among estimators associated to different model selection strategies. Besides, 
it provides an alternative to the $T$-estimators as studied recently in Birg\'e~\citeyearpar{MR2219712}. For illustration, we consider the problems of estimation and (complete) variable selection in various regression settings. 
\end{abstract}
\maketitle
%
\section{Introduction}\label{sect:I}
We consider $k$  independent random measures $N_{1},\ldots,N_{k}$ where the $N_{i}$ are defined on an abstract probability space $(\Omega,\T,\Bbb{P})$ with values in the class of positive measures on measured spaces $(\X_{i},\A_{i},\mu_{i})$. We assume that 
\begin{equation}
\Bbb{E}[N_{i}(A)]=\int_{A} s_{i}d\mu_{i}<+\infty,\quad\mbox{for all }A\in\A_{i}\ {\rm and\ all}\ \ i=1,\ldots,k
\label{Eq-fond}
\end{equation}
where each $s_{i}$ is a nonnegative and measurable function on $\X_{i}$ that we shall call the intensity of $N_{i}$. Equality~\eref{Eq-fond} implies that the $N_{i}$ are a.s. finite measures and that for all measurable and nonnegative functions $f_{i}$ on $\X_{i}$, 
\begin{equation}\label{Eq-f}
\E\cro{\int_{\X_{i}} f_{i}dN_{i}}=\int_{\X_{i}} f_{i}s_{i}d\mu_{i}.
\end{equation}
Our aim is to estimate $s=(s_{1},\ldots,s_{k})$ from the
observation of  $N=(N_{1},\ldots,N_{k})$. We shall set $\X=(\X_{1},\ldots,\X_{k})$, $\A=(\A_{1},\ldots,\A_{k})$, $\mu=(\mu_{1},\ldots,\mu_{k})$
and denote by $\LL$ the cone of nonnegative and measurable functions $t$ of the form $(t_{1},\ldots,t_{k})$ where the $t_{i}$ are positive and integrable functions on $\pa{\X_{i},\A_{i},\mu_{i}}$. For $f=(f_{1},\ldots,f_{k})\in\LL$, we use the notations
\begin{eqnarray*}
\int_{\X}fdN &=& \sum_{i=1}^{k}\int_{\X_{i}}f_{i}dN_{i}\ \ {\rm and}\ \ \int_{\X}fd\mu = \sum_{i=1}^{k}\int_{\X_{i}}f_{i}d\mu_{i}.
\end{eqnarray*}

Throughout, $\LL_{0}$ denotes a known subset of $\LL$ which we assume to contain $s$. This statistical framework we have described allows to deal simultaneously with the more classical ones given below:

\begin{ex}[Density Estimation]\label{ex-densite}
Consider the problem of estimating a density $s$ on $(\X,\A,\mu)$ from the observation of an $n$-sample $X_1,\ldots,X_n$ with distribution $P_{s}=sd\mu$. In order to handle this problem, we shall take $k=1$, $\LL_{0}$ the set of densities on $\pa{\X,\A}$ with respect to $\mu$ and $N=n^{-1}\sum_{i=1}^{n}\delta_{X_{i}}$.
\end{ex}

\begin{ex}[Estimation of marginals]\label{ex-marginal}
Let $X_{1},\ldots,X_{n}$ be independent random variables with values in the measured spaces $(\X_{1},\A_{1},\mu_{1}),\ldots,(\X_{n},\A_{n},\mu_{n})$ respectively. We assume that for all $i$, $X_{i}$ admits a density $s_{i}$ with respect to $\mu_{i}$ and our aim is to estimate $s=(s_{1},\ldots,s_{n})$ from the observation of of $X=(X_{1},\ldots,X_{n})$. 
We shall deal with this problem by taking $k=n$ and $N_{i}=\delta_{X_{i}}$ for $i=1,\ldots,n$. Note that this setting includes as a particular case that of the regression framework
\[
X_{i}=f_{i}+\eps_{i},\ i=1,\ldots,n
\] 
where the $\eps_{i}$ are i.i.d. random variables with a known distribution. The problem of estimating the densities of the $X_{i}$ then amounts to estimating the shift parameter $f=(f_{1},\ldots,f_{n})$. 
\end{ex}

\begin{ex}[Estimating the intensity of a Poisson process]\label{ex-poisson}
Consider the problem of estimating the intensity $s$ of a possibly inhomogeneous Poisson process $N$ on a measurable space $(\X,\A)$. We shall assume that $s$ is integrable. This statistical setting is a particular case of our general one by taking $k=1$ and $\LL_{0}=\LL$. 
\end{ex}
Other examples will be introduced later on. 

Throughout, we shall deal with estimators with values in  $\LL_{0}$ and to measure their risks, endow $\LL_{0}$ with the distance $H$  defined for $t,t'$ in  $\LL_{0}$ by 
\[
H^2(t,t')={1\over 2}\int_{\X}\pa{\sqrt{t}-\sqrt{t'}}^2d\mu={1\over 2}\sum_{i=1}^{k}\int_{\X_{i}}\pa{\sqrt{t_{i}}-\sqrt{t'_{i}}}^2d\mu_{i}.
\]
When $k=1$ and $t,t'$ are densities with respect to $\mu$, $H$ is merely the Hellinger distance between the corresponding probabilities. Given an  estimator $\hat{s}$ of $s$, i.e.\ a measurable function of $N$ with $\hat{s}\in\LL_{0}$, we define its risk by $\E\left[H^2(s,\hat{s})\right]$.  

Let us now give an account of our estimation strategy. We consider an at most countable family $\ac{S_{m},\ m\in\M}$ of subsets of $\LL_{0}$, that we shall call {\it models}, and a family of positive weights $\ac{\Delta_{m}, m\in\M}$ on these satisfying 
\[
\Sigma=\sum_{m\in\M}e^{-\Delta_{m}}<+\infty.
\]
When $\Sigma=1$, the $\Delta_{m}$ define a prior distribution on the family of models and give thus a Bayesian flavor to the procedure. Then, we assume that we have at disposal a collection $\ac{\hat s_{\lambda},\ \lambda\in \Lambda}$ of estimators of $s$ based on $N$ with values in $\SS=\bigcup_{m\in\M}S_{m}$. We mean that each estimator $\hat s_{\lambda}$ belongs to some $S_{m}$ among the family, the index $m=\hat m(\lambda)$ being possibly random depending on the observation $N$. The index set $\Lambda$ need not be countable even though we shall assume so in order to avoid measurability problems. However, the reader can check that the cardinality of $\Lambda$ will play no role in our results. Our aim is to select some $\hat \lambda$ among $\Lambda$, on the basis of the same observation $N$, in such a way that the risk of estimator $\tilde s=\hat s_{\hat \lambda}$ is as close as possible to $\inf_{\lambda\in\Lambda}\E\cro{H^{2}\pa{s,\hat s_{\lambda}}}$. More precisely,  the results we get have the following form
\begin{equation}\label{oracle}
C\E\cro{H^{2}(s,\tilde s)}\le \inf_{\lambda\in\Lambda}\ac{\E\cro{H^{2}\pa{s,\hat s_{\lambda}}}+\tau\E\cro{D_{\hat m(\lambda)}\vee \Delta_{\hat m(\lambda)}}}+\tau\Sigma^{2},
\end{equation}
where 
\begin{itemize}
\item the number $C$ is a positive universal constant;
\item the number $\tau$ is a scaling parameter depending on the statistical framework ($\tau=1/n$ in the density case and $\tau=1$ in the case of Example~\ref{ex-marginal});
\item the numbers $D_{m}$ measure the massiveness (in some suitable sense) of the models $S_{m}$ (typically, $D_{m}$ corresponds to its metric dimension to be defined later on).
\end{itemize}
In Inequality~\eref{oracle}, the element $\hat m(\lambda)$ corresponds to an arbitrary element (chosen by the statistician) among the random subset $\M(\hat s_{\lambda})$ defined by 
\[
\M(\hat s_{\lambda})=\ac{m\in\M,\ \hat s_{\lambda}\in S_{m}}.
\]
Of course, a minimizer of $D_{m}\vee \Delta_{m}$ among those $m$ in $\M(\hat s_{\lambda})$ provides a natural choice for $\hat m(\lambda)$ since it minimizes the right-hand side of~\eref{oracle}. Other choices are possible. For example if for some deterministic $m\in\M$, $\hat s_{\lambda}$ belongs to some $S_{m}$ with probability one, it is convenient to take $\hat m(\lambda)=m$. This is in general the case in the context of model selection for which one associates to each model $S_{m}$ a single estimator, denoted $\hat s_{m}$ rather than $\hat s_{\lambda}$, with values in $S_{m}$. Then, by taking $\Lambda=\M$, Inequality~\eref{oracle} takes the more usual form
\begin{equation}\label{mstypique}
C\E\cro{H^{2}(s,\tilde s)}\le \inf_{m\in\M}\ac{\E\cro{H^{2}\pa{s,\hat s_{m}}}+\tau \pa{D_{m}\vee \Delta_{m}\vee 1}}
\end{equation}
where $C$ depends on $\Sigma$ only. 

In the present paper, our purpose is to go beyond the classical model selection scheme by allowing the family of estimators to take their values in a random model, depending on $N$, among the collection $\ac{S_{m},\ m\in\M}$. Using the same observation $N$, our selection procedure is based on a comparison pair by pair of the estimators $\hat s_{\lambda}$. We do so by mean of a penalized criterion based on an estimation of the distance $H$ of each estimator to the true $s$. From these pairwise comparisons, we use the selection device inspired from Birg\'e~\citeyearpar{MR2219712} and Baraud and Birg\'e~\citeyearpar{MR2449129} to select our estimator $\tilde s$ among the family $\ac{\hat s_{\lambda},\ \lambda\in\Lambda}$. 

Because of these comparisons pair by pair, our procedure is all the more difficult to implement that the cardinality of $\Lambda$ is large. For example, if one tries to estimate a density by an histogram and aims at finding a ``good" partition among a family $\Lambda$ of candidate ones, these comparisons will be time consuming and practically almost useless if $|\Lambda|$ is too large. Nevertheless, one can take advantage that our procedure allows to deal with random partitions $m$ in view of reducing the family $\Lambda$ to those $m$ selected from the data by an appropriate algorithm such as CART for example. From this point of view, our approach can be seen (at least theoretically) as an alternative to resampling procedures (such as V-fold cross-validation, bootstrap,...).

The starting point of this paper originates from a series of papers by Lucien Birg\'e (Birg\'e~\citeyearpar{MR2219712}, Birg\'e~\citeyearpar{Birge-Poisson} and Birg\'e~\citeyearpar{Birge-L2}) providing a new perspective on estimation theory. His approach relies on ideas borrowed from old papers by Le Cam~\citeyearpar{MR0334381}, Le Cam~\citeyearpar{MR0395005}, Birg\'e~\citeyearpar{MR722129}, Birg\'e~\citeyearpar{MR764150}, Birg\'e~\citeyearpar{MR762855}, showing how to derive good estimators  from families of robust tests between simple hypotheses, and also more recent ones about complexity and model selection such as Barron and Cover~\citeyearpar{MR1111806} and Barron, Birg\'e and Massart~\citeyearpar{MR1679028}. The resulting estimator is called a $T$-estimator ($T$ for test) and its construction, detailed in Birg\'e~\citeyearpar{MR2219712}, relies on a good discretization of the models. A nice feature of those $T$-estimators lies in the fact that they require very few assumptions on the collections of models and the parameter set. Our general approach is inspired by this paper even though the procedure we propose is different and allows to consider estimators instead of only discretization points.  

The problem of designing a selection rule solely based on the data in order to choose a ``good" model among a collection of candidate ones is the art of model selection. This approach has been intensively studied in the recent years. For example, Castellan~\citeyearpar{Castellan00}, Castellan~\citeyearpar{Castellan1}, Birg\'e~\citeyearpar{Birge-L2}, Massart~\citeyearpar{MR2319879} (Chapter 7) considered the problem of estimating a density, Reynaud-Bouret~\citeyearpar{MR1981635} and Birg\'e~\citeyearpar{Birge-Poisson} that of estimating the intensity of a Poisson process, and the regression setting has been studied in Baraud~\citeyearpar{MR1777129},  Birg\'e and Massart~\citeyearpar{MR1848946} and Yang~\citeyearpar{Yang99} among other references. Performing model selection for the problem of selecting among histogram-type estimators in the statistical frameworks described in Examples~\ref{ex-densite} and~\ref{ex-poisson} (among others) has been considered in Baraud and Birg\'e~\citeyearpar{MR2449129}. A common feature of all these results on model selection lies in the fact that they hold for specific estimators built on a given model. In the present paper, we shall not specify the estimators $\hat s_{m}$ which can therefore be arbitrary. 

An alternative to model selection is aggregation (or mixing). The basic idea is to design a suitable combination of given estimators in order to outperform each of these separately. This approach can be found in Juditsky and Nemirovski~\citeyearpar{MR1792783}, Nemirovski~\citeyearpar{MR1775640}, Yang~\citeyearpar{MR1790617}, ~\citeyearpar{MR1762904}, ~\citeyearpar{MR1946426}, Tsybakov~\citeyearpar{tsy03}, Wegkamp~\citeyearpar{Wegkamp03}, Bunea, Tsybakov and Wegkamp~\citeyearpar{MR2351101} and Catoni~\citeyearpar{Catoni04} (we refer to his course of Saint Flour which takes back some mixing technics he introduced earlier). When the data are not i.i.d., some nice results of aggregation can be also be found in Leung and Barron~\citeyearpar{MR2242356} for the problem of mixing least-squares estimators of a mean of a Gaussian vector $Y$. In their paper, they assume the components of $Y$ to be  independent with a known common variance. Giraud~\citeyearpar{CGiraud09} extended their results to the case where it is unknown. 

The paper is organized as follows. The basic ideas underlying our approach will be described in Section~\ref{sect-base} and the main results are presented in Section~\ref{sect-main}. In Sections~\ref{sect-histo} and~\ref{sect-discret}, we show how our procedure provides an alternative to these $T$-estimators and histogram-type estimators respectively studied in Birg\'e~\citeyearpar{MR2219712} and Baraud and Birg\'e~\citeyearpar{MR2449129} under the same assumptions. Moreover, we shall also consider in Section~\ref{sect-histo} the case of histogram-type estimators based on random partitions (obtained by an algorithm such as CART for example). In Section~\ref{sect-GLM}, we consider the problem of estimating the mean $s$ of a random vector with nonnegative and independent components (typically the distributions we have in mind are Binomial, Poisson or Gamma). We consider two cases. One corresponds to the situation where $\sqrt{s}=(\sqrt{s_{1}},\ldots,\sqrt{s_{n}})$ is of the form $(F(x_{1}),\ldots,F(x_{n}))$ for some nonnegative function $F$ and points $x_{1},\ldots,x_{n}$ in $[0,1]$. For this problem, we show that the resulting estimator achieves the usual rate of convergence over classes of Besov balls. Alternatively, we consider the situation where $\sqrt{s}$ is a linear combination of predictors $v^{1},\ldots,v^{p}$ the number $p$ being allowed to be larger than $n$. The problem we consider is that of variable selection and we aim at selecting a ``best" subset of predictors in view of minimizing the estimation risk. Section~\ref{sect-regression} is devoted to the regression framework as described in Example~\ref{ex-marginal}. We consider there the problem of complete variable selection when the errors are not Gaussian nor sub-Gaussian which, to our knowledge, is new. In the opposite, the Gaussian case has been intensively studied in the recent years. It has been the usual statistical setting for justifying the use of  numerous procedures among which Birg\'e and Massart~\citeyearpar{MR1848946}, Tibshirani~\citeyearpar{MR1379242} with the Lasso, Efron {\it et al}~\citeyearpar{MR2060166} for LARS, Candes and Tao~\citeyearpar{MR2382644} for the Dantzig selector and Baraud, Giraud and Huet~\citeyearpar{BaGiHu2009} when the variance of the errors is unknown. As we shall see, our selection procedure requires very mild assumptions on the distribution of the errors (provided that it is known). In particular, we need not assume that the errors admit any finite moment. Finally, Section~\ref{sect:proof} is devoted to the proofs.
  
Throughout, we shall use the following notations. The quantity $|E|$ denotes the cardinal of a finite set $E$. The Euclidean norm of $\R^{n}$ is denoted $\norm{\ }$. We set $\R_{+}^{*}=\R_{+}\setminus\ac{0}$ and for $t\in \R_{+}^{*n}$, denote by $\sqrt{t}$ the vector $\pa{\sqrt{t_{1}},\ldots,\sqrt{t_{n}}}$. Given a closed convex subset $A$ of $\R^{n}$, $\Pi_{A}$ is the projection operator onto $A$.  We set for $t\in\LL_{0}$ and $\F\subset\LL_{0}$, $H(t,\F)=\inf_{f\in \F}H(t,f)$ and for $y>0$, 
\[
\B(t,y)=\ac{t'\in\LL_{0},\ H(t,t')\le y}.
\]
Throughout $z$ denotes some number in the interval $(0,1-1/\sqrt{2})$ to be chosen arbitrarily by the statistician and $C, C',C'',...$ constants that may vary from line to line.

\section{Basic formulas and basic ideas}\label{sect-base}
The aim of this section is to present the basic formulas and ideas underlying our approach. For the sake of simplicity, we shall assume $k=1$ until further notice. For $t\in\LL_{0}$, we define 
\[
\rho(s,t)=\int_{\X}\sqrt{st}\ d\mu.
\] 
This quantity corresponds to the Hellinger affinity whenever $s$ and $t$ are densities. Note that $H^{2}(s,t)$ is related to $\rho(s,t)$ by the formula
\[
2H^{2}(s,t)=\int_{\X}sd\mu+\int_{\X}td\mu-2\rho(s,t).
\]
Throughout, $t,t'$ will denote two elements of $\LL_{0}$ one should think of as estimators of $s$. One would  prefer $t'$ to $t$ if $H^{2}(s,t')$ is smaller than $H^{2}(s,t)$ or equivalently if 
\[
\cro{\rho(s,t')-{1\over 2}\int_{\X}t'd\mu}-\cro{\rho(s,t)-{1\over 2}\int_{\X}td\mu}\ge 0.
\]
Since $\int_{\X}td\mu$ and $\int_{\X}t'd\mu$ are both known, deciding whether $t'$ is preferable to $t$ amounts to estimating $\rho(s,t)$ and $\rho(s,t')$ in a suitable way. In the following sections, we present the material that will enable us to estimate these quantities on the basis of the observation $N$.

\subsection{An approximation of $\rho(.,.)$}
We start with the following variational formula. 
\begin{prop}\label{var}
Let $S$ be a subset of $\LL_{0}$ containing  $s$. For all $t\in \LL_{0}$, we have 
\[
\rho(s,t)=\inf_{r\in S}\rho_{r}(sd\mu,t)
\]
where, for a measure $\nu$ on $(\X,\A)$,
\begin{equation}\label{def-rhoQ}
\rho_{r}(\nu,t)={1\over 2}\cro{\rho(t,r)+\int_{\X}\sqrt{{t\over r}}d\nu}\le +\infty
\end{equation}
(using the conventions $0/0=0$ and $a/0=+\infty$ for all $a>0$). Besides, the infimum is achieved for $r=s$.
\end{prop}

\begin{proof}
With the above conventions, note that for all nonnegative numbers $x,y$, $2\sqrt{x}\le \sqrt{y}+x/\sqrt{y}$. By applying this inequality with $x=st$, $y=rt$, the result follows by integration with respect to $\mu$. Besides, equality holds for $r=s$.
\end{proof}

It follows from the above proposition that, for a given $r\in\LL_{0}$, $\rho_{r}(sd\mu,t)$ approximates $\rho(s,t)$ from above. In fact, we can make this statement a little bit more precise.

\begin{prop}\label{var2}
Let $s,t,r\in\LL_{0}$. We have, 
\[
\rho_{r}\pa{sd\mu,t}-\rho(s,t)={1\over 2}\int_{\X}\sqrt{{t\over r}}\pa{\sqrt{s}-\sqrt{r}}^{2}d\mu.
\]
If $r=(t+t')/2$ with $t'\in\LL_{0}$, then
\begin{equation}\label{rhoM-approx}
0\le \rho_{r}\pa{sd\mu,t}-\rho(s,t)\le {1\over \sqrt{2}}\cro{H^{2}(s,t)+H^{2}(s,t')}.
\end{equation}
\end{prop}

\begin{proof} It follows from the definition of $\rho_{r}$ that 
\begin{eqnarray*}
2\cro{\rho_{r}(sd\mu,t)-\rho(s,t)}&=& \int_{\X}\sqrt{tr}\ d\mu+\int_{\X}\sqrt{{t\over r}}\ sd\mu-2\int_{\X}\sqrt{st}\ d\mu\\
&=& \int_{\X}\sqrt{{t\over r}}\pa{\sqrt{s}-\sqrt{r}}^{2}d\mu.
\end{eqnarray*}
For the second part, note that $(t/r)(x)\le 2$ for all $x\in\X$ and therefore $
\rho_{r}(sd\mu,t)-\rho(s,t)\le \sqrt{2}\ H^{2}(s,r)$. It remains to bound $\ H^{2}(s,r)$ from above. The concavity of the map $t\mapsto \sqrt{t}$ implies that $\rho(s,r)\ge\cro{\rho(s,t)+\rho(s,t')}/2$ and therefore $2H^{2}(s,r)\le H^{2}(s,t)+H^{2}(s,t')$, which leads to the result.
\end{proof}

The important point about Proposition~\ref{var2} (more precisely Inequality~\eref{rhoM-approx}) lies in the fact that the constant $1/\sqrt{2}$ is smaller than 1. This makes it possible to use the (sign of the) difference 
\[
T(sd\mu,t,t')=\cro{\rho_{r}(sd\mu,t')-{1\over 2}\int_{\X}t'd\mu}-\cro{\rho_{r}(sd\mu,t)-{1\over 2}\int_{\X}td\mu}
\]
with $r=(t+t')/2$ as an alternative benchmark to find the closest element to $s$ (up to a multiplicative constant) among the pair $\pa{t,t'}$. More precisely, we can deduce from Proposition~\ref{var2} the following corollary.

\begin{cor}
If $T(sd\mu,t,t')\ge 0$, then
\[
H^{2}(s,t')\le {\sqrt{2}+1\over \sqrt{2}-1}H^{2}(s,t).
\]
\end{cor}

\begin{proof}
Using Inequality~\eref{rhoM-approx} and the assumption, we have
\begin{eqnarray*}
H^{2}\pa{s,t'}- H^{2}\pa{s,t}&=&\cro{\rho(s,t)-{1\over 2}\int_{\X}td\mu}-\cro{\rho(s,t')-{1\over 2}\int_{\X}t'd\mu}\\
&=& \cro{\rho_{r}(sd\mu,t)-{1\over 2}\int_{\X}td\mu}-\cro{\rho_{r}(sd\mu,t')-{1\over 2}\int_{\X}t'd\mu}\\
&& +\ \rho\pa{s,t}-\rho_{r}\pa{sd\mu,t}+\ \rho_{r}\pa{sd\mu,t'}-\rho\pa{s,t'}\\
&\le& {1\over \sqrt{2}}\cro{H^{2}\pa{s,t}+H^{2}\pa{s,t'}}
\end{eqnarray*}
which leads to the result. 
\end{proof}

\subsection{An estimator of $\rho_{r}(.,.)$}
Throughout, given $t,t'\in\LL_{0}$, we set 
\[
r={t+t'\over 2}\in\LL_{0}.
\]
The superiority of the quantity $\rho_{r}\pa{sd\mu,t}$ over 
$\rho\pa{s,t}$ lies in the fact that the former can easily be estimated by its empirical counterpart, namely 
\begin{equation}\label{def-rhohat}
\rho_{r}\pa{N,t}={1\over 2}\cro{\rho(t,r)+\int \sqrt{{t\over r}}dN}.
\end{equation}
Note that $\rho_{r}\pa{N,t}$ is an unbiased estimator of $\rho_{r}\pa{sd\mu,t}$ because of~\eref{Eq-f}. Consequently, a natural way of deciding which between $t$ and $t'$ is the closest to $s$ is to consider the test statistics 
\[
T(N,t,t')=\cro{\rho_{r}(N,t')-{1\over 2}\int_{\X}t'd\mu}-\cro{\rho_{r}(N,t)-{1\over 2}\int_{\X}td\mu}.
\]

Replacing the ``ideal" test statistic $T(sd\mu,t,t')$ by its empirical counterpart leads to an estimation error given by the process $Z(N,.,.)$ defined on $\LL_{0}^{2}$ by
\begin{eqnarray*}
Z(N,t,t')&=&T(N,t,t')-T(sd\mu,t,t')\\
&=&
\cro{\rho_{r}\pa{N,t'}-\rho_{r}\pa{sd\mu,t'}}-\cro{\rho_{r}\pa{N,t}-\rho_{r}\pa{sd\mu,t}}\\
&=& \int_{\X}\psi(t,t',x)dN-\int_{\X}\psi(t,t',x)sd\mu
\end{eqnarray*}
where $\psi(t,t',x)$ is the function on $\LL_{0}^{2}\times\X$ with values in $[-1/\sqrt{2},1/\sqrt{2}]$ given by
\begin{equation}\label{def-psi}
\psi(t,t',x)={1\over \sqrt{2}}\cro{\ \sqrt{{1\over 1+t(x)/t'(x)}}-{\sqrt{1\over 1+t'(x)/t(x)}}\ }.
\end{equation}
The study of the empirical process $Z(N,.,.)$ over the product space $S\times S'$ is at the heart of our technics. 

\subsection{The multidimensional case $k>1$} 
In the multidimensional case, the same results can be obtained by reasoning component by component. More precisely, the formulas of the above sections extend by using the convention that for all $k$-uplets $\nu=(\nu_{1},\ldots,\nu_{k})$ of measures on $(\X_{1},\A_{1}),\ldots,(\X_{k},\A_{k})$ respectively,
\[
\int_{\X}\phi(s,t,t',r)d\nu=\sum_{i=1}^{k}\int_{\X_{i}}\phi(s_{i},t_{i},t'_{i},r_{i})d\nu_{i},
\]
whatever the functions $s,t,t',r\in\LL_{0}$ and mappings $\phi$ from $\R_{+}^{4}$ into $\R$. 

\section{The main results}\label{sect-main}
Throughout this section, we consider an at most countable index set $\M$ and 
a family $\ac{S_{m},\ m\in\M}$  of nonvoid subsets of $\LL_{0}$, we shall refer to as models. Besides, we assume we have at disposal an at most countable family $\ac{\hat s_{\lambda},\ \lambda\in\Lambda}$ of estimators of $s$ based on $N$ with values in $\SS=\bigcup_{m\in\M} S_{m}$. In particular, to each $\lambda\in\Lambda$ corresponds an estimator $\hat s_{\lambda}$ together with a (possibly random) index $\hat m(\lambda)\in\M$ such that $\hat s_{\lambda}\in S_{\hat m(\lambda)}$. Setting for $t\in\SS$,   
\[
\M\pa{t}=\ac{m\in\M,\ t\in S_{m}}
\]
we therefore have $\hat m(\lambda)\in\M(\hat s_{\lambda})$. We associate a nonnegative weight $\Delta_{m}$ to each $m\in\M$ and assume that 
\begin{equation}\label{sigma}
\Sigma=\sum_{m\in \M}e^{-\Delta_{m}}<+\infty\ \ \text{and}\ \ \Delta_{m}\ge 1\ \ \text{for all}\ \ m\in\M.
\end{equation}
The condition $\Delta_{m}\ge 1$ for all $m\in\M$ is only required to simplify the presentation of our results. 

As already mentioned in the introduction, our aim is to select some estimator among the family $\ac{\hat s_{\lambda},\ \lambda\in\Lambda}$ in order to achieve the smallest possible risk. We shall distinguish between two situations. 

\subsection{Direct selection}
Let $\tau,\gamma$ be positive numbers. We consider the following selection procedure 
\begin{proc}\label{P1}
Let $\pen$ be some penalty function mapping $\SS$ into $\R_{+}$. Given a pair $(\hat s_{\lambda},\hat s_{\lambda'})$ such that $\hat s_{\lambda}\neq\hat s_{\lambda'}$, we consider the test statistic
\begin{eqnarray}
\gT(N,\hat s_{\lambda},\hat s_{\lambda'})&=&\cro{\rho_{r}(N,\hat s_{\lambda'})-{1\over 2}\int_{\X}\hat s_{\lambda'}d\mu-\pen(\hat s_{\lambda'})}\label{critere}\\
&&\ \ -\cro{\rho_{r}(N,\hat s_{\lambda})-{1\over 2}\int_{\X}\hat s_{\lambda}d\mu-\pen(\hat s_{\lambda})}\nonumber
\end{eqnarray}
where $r=(\hat s_{\lambda}+\hat s_{\lambda'})/2$ and $\rho_{r}(N,.)$ is given by~\eref{def-rhohat}. We set 
\[
\EE(\hat s_{\lambda})=\ac{\hat s_{\lambda'},\ \gT(N,\hat s_{\lambda},\hat s_{\lambda'})\ge 0}
\]
and note that either $\hat s_{\lambda}\in \EE(\hat s_{\lambda'})$ or $\hat s_{\lambda'}\in \EE(\hat s_{\lambda})$ since $\gT(N,\hat s_{\lambda},\hat s_{\lambda'})=-\gT(N,\hat s_{\lambda'},\hat s_{\lambda})$. Then, we define 
\[
\D(\hat s_{\lambda})=\sup\ac{H^{2}\pa{\hat s_{\lambda},\hat s_{\lambda'}}\telque\  \hat s_{\lambda'}\in\EE(\hat s_{\lambda})}\ \ {\rm if}\ \ \EE(\hat s_{\lambda})\neq \varnothing
\]
and $\D(\hat s_{\lambda})=0$ otherwise. Finally, we select $\hat \lambda$ among $\Lambda$ as any element satisfying 
\[
\D(\hat s_{\hat \lambda})\le \D(\hat s_{\lambda})+\tau,\ \ \ \forall \lambda\in\Lambda.
\]
\end{proc}

For $(t,t')\in\LL_{0}^{2}$ and $y>0$, let us set
\[
w^{2}(t,t',y)=\cro{{H^{2}\pa{s,t}+H^{2}\pa{s,t'}}}\vee y^{2}.
\]
We assume the following.
\begin{ass}[$\tau,\gamma$]\label{hypo}
For all pairs $(m,m')\in\M^{2}$, there exist positive numbers $d_{m},d_{m'}$ such that for all $\xi>0$ and $y^{2}\ge \tau\pa{d_{m}\vee d_{m'}+\xi}$,
\[
\P\cro{\sup_{(t,t')\in S_{m}\times S_{m'}}{Z(N,t,t')\over w^{2}(t,t',y)}\ge z}\le \gamma e^{-\xi}.
\]
\end{ass} 
This assumption means that for $\xi$ large enough the error process $Z(N,t,t')$ is uniformly controlled by $w^{2}(t,t',y)$ over $S_{m}\times S_{m'}$ with probability close to 1. Under suitable assumptions, the quantities $d_{m}$ measure in some sense the massiveness of the $S_{m}$. For example, if  $S_{m}$ is the linear span of piecewise constant function on each element of a partition $m$ of $\X$, then $d_{m}$ is merely proportional to the cardinality of $m$. If $S_{m}$ is a discrete subset $\LL_{0}$, $d_{m}$ is related to its metric dimension (in a sense to be specified later on). 

We obtain the following result.
\begin{thm}\label{main}
Let $\tau,\gamma$ be numbers and $\ac{\Delta_{m},\ m\in\M}$ a family of nonnegative numbers satisfying~\eref{sigma}. Under Assumption~\ref{hypo},  choose $\tilde s=\hat s_{\hat \lambda}$ among the family $\ac{\hat s_{\lambda},\ \lambda\in\Lambda}$  according to Procedure~\ref{P1} with  $\pen$ satisfying
\begin{equation}\label{def-pen}
\pen(t)\ge z\tau\inf\ac{d_{m}+\Delta_{m},\ m\in \M\pa{t}}\ \ \forall t\in\SS.
\end{equation}
Then, for all $\xi>0$,
\[
\P\cro{H^{2}\pa{s,\tilde s}\ge C_{1}\cro{\inf_{\lambda\in\Lambda}\cro{H^{2}\pa{s,\hat s_{\lambda}}+\pen\pa{\hat s_{\lambda}}}}+C_{2}\tau\xi}\le \pa{\gamma\Sigma^{2} e^{-\xi}}\wedge 1.
\]
where $C_{1}=C_{1}(z)$ and $C_{2}=C_{2}(z)$ are positive numbers given by~\eref{C1} and~\eref{C2} respectively, depending on the choice of $z$ only. 
\end{thm}
The proof is delayed to Section~\ref{sect-pmain}. 

By integration with respect to $\xi$ we deduce the following risk bound.
\begin{cor}\label{cor1}
Under the assumptions of Theorem~\ref{main}, there exists a constant $C$ depending on $z$ only such that
\begin{eqnarray*}
C\E\cro{H^{2}\pa{s,\tilde s}}&\le& \E\cro{\inf_{\lambda\in\Lambda}\ac{H^{2}\pa{s,\hat s_{\lambda}}+\pen(\hat s_{\lambda})}}+\tau\cro{(\gamma\Sigma^{2})\vee 1}\\
&\le& \inf_{\lambda\in\Lambda}\ac{\E\cro{H^{2}\pa{s,\hat s_{\lambda}}+\pen(\hat s_{\lambda})}}+\tau\cro{(\gamma\Sigma^{2})\vee 1}.
\end{eqnarray*}

In particular, if equality holds in~\eref{def-pen},
\begin{eqnarray}
\E\cro{H^{2}\pa{s,\tilde s}}&\le& C'\inf_{\lambda\in\Lambda}\ac{\E\cro{H^{2}\pa{s,\hat s_{\lambda}}}+\E\cro{v^{2}(\hat s_{\lambda})}}\label{joli}
\end{eqnarray}
where, for all $\lambda\in\Lambda$,
\begin{equation}\label{v}
v^{2}(\hat s_{\lambda})=\tau\cro{\inf_{m\in\M(\hat s_{\lambda})}d_{m}\vee \Delta_{m}}\le \tau\pa{d_{\hat m(\lambda)}\vee \Delta_{\hat m(\lambda)}}
\end{equation}
and $C'$ is a constant depending on $z,\gamma$ and $\Sigma$.
\end{cor}
Inequality~\eref{joli} compares the risk of the resulting estimator $\tilde s$ to those of the $\hat s_{\lambda}$ plus an additional term $\E\cro{v^{2}(\hat s_{\lambda})}$. If $\hat s_{\lambda}$ belongs to $S_{m}$  with probability 1,
\begin{equation}\label{born1}
v^{2}(\hat s_{\lambda})\le \tau \pa{d_{m}\vee \Delta_{m}}.
\end{equation}
We emphasize that~\eref{born1} does not take into account the complexity of the collection of estimators $\ac{\hat s_{\lambda},\ \lambda\in\Lambda}$ itself. In particular, if for all $\lambda\in\Lambda$, $\hat s_{\lambda}$ belongs to a same model $S_{m}$ with probability 1, then by taking $\M=\ac{m}$ and $\Delta_{m}=1$, we obtain for $\tilde s$ the following risk bound 
\[
\E\cro{H^{2}(s,\tilde s)}\le C'\ac{\inf_{\lambda\in\Lambda}\E\cro{H^{2}\pa{s,\hat s_{\lambda}}}+\tau \pa{d_{m}\vee 1}}
\]
no matter how large the collection of $\hat s_{\lambda}$ is. 

\subsection{Indirect selection}
Let $\tau,M$ be some positive numbers. Throughout this section, we assume that for some nonnegative numbers $a,b,c$, the measure $N$ satisfies the following.

\begin{ass}[$a,b,c$]\label{bern}
For all $y,\xi> 0$ 
\[
\sup_{t,t'\in \B(s,y)}\P\cro{Z(N,t,t')>\xi}\le b\exp\cro{-{a\xi^{2}\over y^{2}+c\xi}}.
\]
\end{ass}

This assumption is satisfied in the following cases.

\begin{prop}\label{Examples}
Assumption~\ref{bern} holds with $a=n^{2}/6$, $b=1$ and $c=n\sqrt{2}/6$ for Example~\ref{ex-densite}, with $a=1/6$, $b=1$ and $c=\sqrt{2}/6$ for Example~\ref{ex-marginal} and with $a=1/12$, $b=1$ and $c=\sqrt{2}/36$ for Example~\ref{ex-poisson}.
\end{prop}

The proof of the proposition is delayed to Section~\ref{preuves-ex}. 

In order to select among the family of estimators $\ac{\hat s_{\lambda},\ \lambda\in\Lambda}$, we introduce an auxiliary family $\ac{\S_{m},\ m\in\M}$ of discrete subsets of $\LL_{0}$ satisfying the following assumption. 

\begin{ass}[$\tau,M$]\label{h-discret}
For all $m\in\M$ and $s\in\LL_{0}$, there exists $\eta_{m}\ge 1/2$ such that
\[
\ab{\S_{m}\cap\B(s,r\sqrt{\tau})}\le M\exp\pa{{r^{2}\over 2}},\ \ \forall r\ge 2\eta_{m}.
\]
\end{ass}
As we shall see, the parameter $\eta_{m}^{2}$ is convenient to measure the massiveness of the discrete set $\S_{m}$. It is related to a metric dimension (in a sense to be specified later on). 

Assumptions~\ref{bern} and~\ref{h-discret} are related to our former Assumption~\ref{hypo} by the following result.  
\begin{lemma}\label{lem1}
If Assumptions~\ref{bern} and~\ref{h-discret} hold with $\tau=4(2+cz)/(az^{2})$ then the collection of models $\ac{\S_{m},m\in\M}$ satisfy Assumption~\ref{hypo} with $\gamma=bM^{2}$ and $d_{m}=4\eta_{m}^{2}$ for all $m\in\M$. 
\end{lemma}

Consider now the following selection procedure.
\begin{proc}\label{P2}
Let $\pen$ be some penalty function from $\S=\bigcup_{m\in\M}\S_{m}$ into $\R_{+}$. To each $\lambda\in\Lambda$, associate the auxiliary estimator $\tilde s_{\lambda}$ as any element of $\S$ satisfying 
\[
H^{2}\pa{\hat s_{\lambda},\tilde s_{\lambda}}+\pen(\tilde s_{\lambda})\le A(\hat s_{\lambda},\S)+\tau
\]
where 
\[
A(\hat s_{\lambda},\S)=\inf_{t\in\S}\cro{H^{2}\pa{\hat s_{\lambda},t}+\pen(t)}
\]
Select $\tilde \lambda$ among $\Lambda$ by using Procedure~\ref{P1} with the family of estimators $\ac{\tilde s_{\lambda},\ \lambda\in\Lambda}$. Finally, select $\hat \lambda$ as any element of $\Lambda$ such that 
\[
H^{2}(\hat s_{\hat \lambda},\tilde s_{\tilde \lambda})\le \inf_{\lambda\in\Lambda}H^{2}(\hat s_{\lambda},\tilde s_{\tilde \lambda})+\tau.
\]
\end{proc}

The following holds.
\begin{thm}\label{main2}
Let $M$ be a positive number and $\ac{\Delta_{m},\ m\in\M}$ a family of numbers satisfying~\eref{sigma}. Assume that Assumption~\ref{bern} and~\ref{h-discret} hold with $\tau=4(2+cz)/(az^{2})$. Let $\tilde s=\hat s_{\hat \lambda}$ be the estimator obtained by selecting $\hat \lambda$ according to Procedure~\ref{P2} with 
 \begin{equation}\label{def-pen2}
\pen(t)\ge z\tau\inf_{m\in\M(t)}\pa{4\eta_{m}^{2}+\Delta_{m}}\ \ \forall t\in\S.
\end{equation}
Then, for all $\xi>0$,
\[
\P\cro{H^{2}\pa{s,\tilde s}\ge C\cro{\inf_{\lambda\in\Lambda}\pa{H^{2}\pa{s,\hat s_{\lambda}}+A(\hat s_{\lambda},\S)}+\tau\xi}}\le \pa{bM^{2}\Sigma^{2} e^{-\xi}}\wedge 1,
\]
and
\begin{eqnarray*}
C'\E\cro{H^{2}\pa{s,\tilde s}}&\le& \E\cro{\inf_{\lambda\in\Lambda}\ac{H^{2}\pa{s,\hat s_{\lambda}}+A(\hat s_{\lambda},\S)}}+\tau\cro{(bM^{2}\Sigma^{2})\vee 1}\\
&\le& \inf_{\lambda\in\Lambda}\ac{\E\cro{H^{2}\pa{s,\hat s_{\lambda}}+A(\hat s_{\lambda},\S)}}+\tau\cro{(bM^{2}\Sigma^{2})\vee 1}.
\end{eqnarray*}
where $C,C'$  are positive numbers depending on $z$ only. 
\end{thm}

The risk bound we get involves the quantity $A(\hat s_{\lambda},\S)$ which depends on the approximation property of $\S$ with respect to the (random) family $\ac{\hat s_{\lambda}, \lambda\in\Lambda}\subset \SS$. In the favorable situation where the $\hat s_{\lambda}$ take their values in $\S$ and if equality holds in~\eref{def-pen2}, then 
\[
A(\hat s_{\lambda},\S)\le \pen(\hat s_{\lambda})\le z\tau\pa{4\eta^{2}_{\hat m(\lambda)}+\Delta_{\hat m(\lambda)}}.
\]
In a more general case, one needs to choose $\S$ to possess good approximation properties with respect to the elements of $\SS$ in order to keep the quantity $A(\hat s_{\lambda},\S)$ as small as possible for all $\lambda\in\Lambda$. To ensure such a property, it is convenient to choose $\S_{m}$ as a suitable discretization of $S_{m}$ for all $m\in\M$. 

\begin{defi}
Let $S$ be a subset of $(\LL_{0},H)$ and $\eps$ some positive number. We shall say that $\S$ is an $\eps$-net for $S$ if $\S\subset S$ and if for all $t\in S$, there exists $t'\in\S$ such that $H(t,t')\le \eps$. For nonnegative numbers $M,D$, we shall specify that $\S$ is an $(M,\eps,D)$-net for $S$ if for all $s\in \LL_{0}$ and $r\ge 2\eps$,
\begin{equation}\label{nbdept}
\ab{\ac{t\in\S,\ H(s,t)\le r}}\le M\exp\cro{D\pa{{r\over \eps}}^{2}}.
\end{equation}
\end{defi}

The parameter $D$ corresponds to an upper bound to what is usually called the metric dimension of $S$ (we refer to Birg\'e~\citeyearpar{MR2219712}, Definition~6). Under suitable assumptions and provided that the $\eps$-net has been suitably chosen, the metric dimension $D$ of $S$ provides an upper bound (up to a suitable renormalisation) for the minimax estimation rate over $S$. In many cases of interest, it turns that $D$ actually provides the right order of magnitude but, unfortunately, not always. For a complete discussion with examples and counter-examples on the connection between metric dimensions and minimax estimation rates we refer the reader to Birg\'e~\citeyearpar{MR722129} and Yang and Barron~\citeyearpar{MR1742500}. 

We deduce from Theorem~\ref{main2} the following corollary.

\begin{cor}\label{cor2ter}
Let $M$ be a positive number and $\ac{\Delta_{m},\ m\in\M}$ a family of nonnegative numbers satisfying~\eref{sigma}. Assume that Assumption~\ref{bern} holds and that for $m\in\M$, $\S_{m}$ is a $(M,\eta_{m}\sqrt{\tau},D_{m})$-net for $S_{m}$ with $\tau=4(2+cz)/(az^{2})$ and $\eta_{m}^{2}=2(D_{m}\vee 1/8)$. If equality holds in~\eref{def-pen2}, the estimator $\tilde s$ defined in Theorem~\ref{main2} satisfies
\begin{eqnarray}
\E\cro{H^{2}\pa{s,\tilde s}}&\le& C\inf_{\lambda\in\Lambda}\ac{\E\cro{H^{2}\pa{s,\hat s_{\lambda}}}+\tau\E\cro{D_{\hat m(\lambda)}\vee\Delta_{\hat m(\lambda)}}}\label{jolib}
\end{eqnarray}
where $C$ is a constant depending on $z,M$ and $\Sigma$ only.
\end{cor}

Since the statistician is free to choose $\hat m(\lambda)$ any element among $\M(\hat s_{\lambda})$, a natural choice in view of minimizing~\eref{jolib} is to take it as (any) minimizer of  $D_{m}\vee\Delta_{m}$ among those $m\in\M(\hat s_{\lambda})$.

If one considers a family of estimators $\ac{\hat s_{m},\ m\in\M}$ (here $\Lambda=\M$) such that $\hat s_{m}$ belongs to $S_{m}$ with probability one, we deduce from Corollary~\ref{cor2ter} that the estimator $\tilde s=\hat s_{\hat m}$ satisfies, 
\begin{equation}\label{eqz}
\E\cro{H^{2}\pa{s,\tilde s}}\le C\inf_{m\in\M}\ac{\E\cro{H^{2}\pa{s,\hat s_{m}}}+\tau\pa{D_{m}\vee\Delta_{m}}}.
\end{equation}
Moreover, if for some universal constant $c>0$ the estimators $\hat s_{m}$ satisfy 
\[
\E\cro{H^{2}\pa{s,\hat s_{m}}}\ge c\tau D_{m},\ \forall s\in\LL_{0}\ \ \forall m\in\M,
\]
then~\eref{eqz} shows that $\tilde s$ satisfies the oracle-type inequality
\[
\E\cro{H^{2}\pa{s,\tilde s}}\le C'\inf_{m\in\M}\ac{\E\cro{H^{2}\pa{s,\hat s_{m}}}\vee (\tau \Delta_{m})}.
\]

\section{Selecting among histogram-type estimators}\label{sect-histo}
In this section we assume that $\M$ is a family of partitions of $\X$ and for $m\in\M$, $S_{m}$ the set gathering the elements of $\LL_{0}$ which are piecewise constant on each element of the partition $m$, that is 
\[
S_{m}=\ac{\sum_{I\in m}a_{I}\1_{I}\telque\ (a_{I})_{I\in m}\in\R^{|m|}}\bigcap \LL_{0}.
\]
We shall therefore consider a family $\ac{S_{m},\ m\in\M}$ of such models and $\ac{\hat s_{\lambda},\ \lambda\in\Lambda}$ a family of estimators of the form $\sum_{I\in \hat m}\hat a_{I}\1_{I}$, the values $\hat a_{I}$ and the partition $\hat m\in\M$ being allowed to be random depending on the observation $N$. 

Throughout this section, we assume that $k=1$. The applications we have in mind include Examples~\ref{ex-densite} and~\ref{ex-poisson} and also the following statistical setting. 
\begin{ex}\label{ex-means}
We observe a vector $X=(X_{1},\ldots,X_{n})$ the components of which are independent and nonnegative with respective means $s_{i}$. Our aim is to estimate $s=(s_{1},\ldots,s_{n})$ on the basis of the observation of $X$. This statistical setting is a particular case of our general one described in Section~\ref{sect:I} by taking $k=1$, $\X=\ac{1,\ldots,n}$, $\A=\PP(\X)$, $\mu$ the counting measure on $(\X,\A)$, $\LL_{0}=\LL$ and $N$ the measure defined for $A\subset \X$ by 
\[
N(A)=\sum_{i\in A}X_{i}.
\]
\end{ex}
Among the distributions we have in mind for the $X_{i}$, we mention the Binomial or Gamma. 

For partitions $m,m'$ of $\X$, we set 
\[
\X^{2}(m)=\sum_{I\in m}\pa{\sqrt{N(I)}-\sqrt{\E(N(I))}}^{2}.
\]
and
\[
m\vee m'=\ac{I\cap I',\ (I,I')\in m\times m'}.
\]

\subsection{The assumptions}
We assume that $N$ satisfies
\begin{ass}\label{H3}
There exists a positive number $\tau$ such that for all $\xi>0$ and all partition $m$ of $\X$
\begin{equation}\label{deviation}
\P\cro{\X^{2}(m)\ge a\pa{|m|+\xi}}\le e^{-\xi}.
\end{equation}
\end{ass}

Besides, we assume that the family of partitions $\M$ satisfies the following
\begin{ass}\label{H4}
There exists $\delta\ge 1$ such that $|m\vee m'|\le \delta \pa{|m|\vee |m'|}$ for all $m,m'\in\M$.
\end{ass}

These two assumptions also appeared in Baraud et Birg\'e~\citeyearpar{MR2449129} as Assumptions H and H' in their Theorem~6. In particular, the following result is proven there
\begin{prop}
Assumption~\ref{H3} holds with $a=200/n$ in the case of Example~\ref{ex-densite}, with $a=6$ in the case of Example~\ref{ex-poisson} and, in the case of Example~\ref{ex-means}, with 
\[
a=3\kappa\pa{1/\sqrt{2}+\sqrt{\pa{{\beta\over \kappa}-{1\over 2}}_{+}}}
\]
provided that for some $\beta\ge 0$ and $\kappa>0$, the $X_{i}$ satisfy for $i=1,\ldots,n$ 
\[
\E\cro{e^{u\pa{X_i-s_i}}}\le \exp\cro{\kappa\frac{u^2s_i}{2(1-u\beta)}}\quad
\mbox{for all }u\in\left[0,\frac{1}{\beta}\right[,
\]
with the (convention $1/\beta=+\infty$ if  $\beta=0$), and
\[
\E\cro{e^{-u\pa{X_i-s_i}}}\le \exp\cro{\kappa \frac{u^2s_i}{2}} \ \quad\mbox{for
all }u\geq 0.
\]
\end{prop}
Throughout this section, we set $\tau=20az^{-2}$.

\subsection{The main result}
\begin{thm}\label{histo}
Assume that Assumptions~\ref{H3} and~\ref{H4} hold and that $\ac{\Delta_{m},\ m\in\M}$ satisfies~\eref{sigma}. Consider a family $\ac{\hat s_{\lambda},\ \lambda\in\Lambda}$ of estimators of $s$ with values in $\SS$. If $\pen$ is such that 
\[
\pen(t)\ge z\tau\inf_{m\in\M(t)}\pa{\delta|m|+\Delta_{m}}\ \ \forall t\in\SS
\] 
the estimator $\tilde s=\hat s_{\hat \lambda}$ selected by Procedure~\ref{P1} satisfies for some constant $C$ depending on $z$ only, 
\[
C\E\cro{H^{2}\pa{s,\tilde s}}\le \E\cro{\inf_{\lambda\in\Lambda}\cro{H^{2}\pa{s,\hat s_{\lambda}}+\pen\pa{\hat s_{\lambda}}}}+\tau(\Sigma^{2}\vee 1).
\]
\end{thm}

The above result holds for any choices of estimators  $\ac{\hat s_{\lambda},\ \lambda\in\Lambda}$ with values in $\SS$. Of special interest are the estimators $\hat s_{m}$ associated to a partition $m$ of $\X$ by the formula
\begin{equation}\label{def-histo}
\hat s_{m}=\sum_{I\in m}{N(I)\over\mu(I)}\1_{I}.
\end{equation}
Since when $\mu(I)=0$, $\E(N(I))=\int_{I}sd\mu=0$ and $N(I)=0$ a.s., the estimator $\hat s_{m}$ is well-defined with the conventions $0/0=0$ and $c/\infty=0$ for all $c>0$. One can prove (we refer to Baraud and Birg\'e~\citeyearpar{MR2449129}) that for all $m\in\M$,
\[
\E\cro{H^{2}(s,\hat s_{m})}\le 4\pa{H^{2}(s,S_{m})+\tau|m|}.
\]
In the following sections, we shall apply Theorem~\ref{histo} in order to choose among a family of such estimators.

\subsection{Model selection}
Let $\M$ be a family partitions of $\X$ and associate to each $m\in\M$, the estimator $\hat s_{m}$ defined by~\eref{def-histo}. We deduce from Theorem~\ref{histo} the following corollary.

\begin{cor}\label{BB}
Assume that Assumptions~\ref{H3} and~\ref{H4} hold and that $\ac{\Delta_{m},\ m\in\M}$ satisfies~\eref{sigma}. Choose $\tilde s=\hat s_{\hat m}$ among $\ac{\hat s_{m},\ m\in\M}$ by using Procedure~\ref{P1} and
\[
\pen(\hat s_{m})=z\tau\pa{\delta|m|+\Delta_{m}}\ \ \forall m\in\M.
\]
Then there exists a constant $C$ depending on $z,\delta$ and $\Sigma$ only such that
\begin{eqnarray*}
C\E\cro{H^{2}\pa{s,\tilde s}}&\le& \inf_{m\in\M}\ac{\E\cro{H^{2}(s,\hat s_{m})}+\tau\pa{|m|\vee \Delta_{m}}}\\
&\le& \inf_{m\in\M}\cro{H^{2}\pa{s,S_{m}}+\tau(|m|\vee \Delta_{m})}.
\end{eqnarray*}
\end{cor}

This corollary recovers the results of Theorem~6 in Baraud and Birg\'e~\citeyearpar{MR2449129} even though the selection procedure is different. The choice of a suitable family $\M$ of partitions is of course a crucial point. It should be chosen in such a way that the family $\ac{S_{m},\ m\in\M}$ possesses good approximation properties with respect to classes of functions $s$ of interest. This point has been discussed in Baraud and Birg\'e~\citeyearpar{MR2449129} (see their Section~3). Another concern is the computational cost. In the case of density estimation, alternative selection procedures based on the minimization of a penalized criterion over families $\M$ generated by an algorithm such as CART (or some related version) can be less time consuming. We refer for example to Blanchard {\it et al}~\citeyearpar{MR2177922} which considers families of partitions associated to some dyadic decision trees. Their algorithm is inspired from that of Donoho~\citeyearpar{MR1474073} in the context of regression in 2D. 

In view of reducing the computation cost of our selection procedure, we extend Corollary~\ref{BB} to the case where the partitions $m$ are possibly random, generated from the data themselves.  

\subsection{Selecting among model selection strategies}
Assume now that each $\lambda\in\Lambda$ is a model selection strategy allowing to choose a partition $\hat m(\lambda)$ among a collection of candidate partitions $\M$. Besides, to each $\lambda\in\Lambda$, associate the estimator $\hat s_{\lambda}=\hat s_{\hat m(\lambda)}$ with $\hat s_{m}$ defined by~\eref{def-histo} for all $m\in\M$. 

By applying Theorem~\ref{histo} to the collection $\ac{\hat s_{\lambda},\ \lambda\in\Lambda}$ we get the following result. 

\begin{cor}\label{BSA}
Assume that Assumptions~\ref{H3},~\ref{H4} hold and that $\ac{\Delta_{m},\ m\in\M}$ satisfies~\eref{sigma}. Choose $\tilde s=\hat s_{\hat \lambda}$ among $\ac{\hat s_{\lambda},\ \lambda\in\Lambda}$ by using Procedure~\ref{P1} and 
\[
\pen(\hat s_{\lambda})=z\tau\pa{\delta|\hat m(\lambda)|+\Delta_{\hat m(\lambda)}}\ \ \forall \lambda\in\Lambda.
\] 
Then, for some constant $C$ depending on $z,\delta$ and $\Sigma$ only, 
\[
C\E\cro{H^{2}\pa{s,\tilde s}}\le \inf_{\lambda\in\Lambda}\E\cro{H^{2}\pa{s,\hat s_{\lambda}}+\tau\pa{|\hat m(\lambda)|\vee\Delta_{\hat m(\lambda)}}}.
\]
\end{cor}

Note that Corollary~\ref{BSA} shows that the risk of $\tilde s$ can be related to those of the $\hat s_{\lambda}$ but gives no hint on the orders of magnitude of the latters. Such a study is beyond the scope of this paper. In density estimation, Lugosi and Nobel~\citeyearpar{MR1394983} tackled this problem by giving sufficient condition on the random partition $\hat m(\lambda)$ to ensure the $\IL_{1}$-consistency of the estimator $\hat s_{\lambda}$, that is, under suitable conditions, they show that
\[
\int_{\X}\ab{s-\hat s_{\lambda}}d\mu\to 0\ a.s. 
\]
as the sample size tends to infinity. Since 
\[
H^{2}(s,\hat s_{\lambda})\le \int_{\X}\ab{s-\hat s_{\lambda}}d\mu,
\]
the same holds for distance $H$ and by dominated convergence, we deduce that $\E\cro{H^{2}(s,\hat s_{\lambda})}$ also tends to 0 as the sample size tends to infinity.

We end this section by giving a simple way of choosing a family of partitions from the data by mean of a contrast. We shall assume for simplicity that $\X=[0,1)$ and consider the family $\M$  of partitions of $[0,1)$ into intervals of the form $[a,b)$ the endpoints of which belong to the regular grid $\ac{k/N,\ k=0,\ldots,N}$ with $N\ge 2$. For such a family, it is easy to check that the choice $\Delta_{m}=|m|\log(N-1)$ ensures that~\eref{sigma} holds with $\Sigma\le e$. In what follows, the notation $m\preceq m'$ for $m,m'\in\M$ means that the partition $m'$ is thinner than $m$ or equivalently that $S_{m}\subset S_{m'}$. Let us now introduce the criterion $\crit(N,t)$ defined for $t\in\SS$ by 
\[
\crit(N,t)=-2\int_{\X}tdN+\int_{\X}t^{2}d\mu.
\]
It is well-known that $\crit(N,.)$ is a contrast on $\SS$ and that if $s$ belongs to $\IL^{2}([0,1),\mu)$, for all $t,t'\in\SS$
\begin{equation}\label{contraste}
\E\cro{\crit(N,t)-\crit(N,t')}=\int_{\X}\pa{s-t}^{2}d\mu-\int_{\X}\pa{s-t'}^{2}d\mu.
\end{equation}
Then, given a partition $m\in\M$ it is natural to associate to $S_{m}$ the estimator obtained by minimizing $\crit(N,t)$ among those $t$ in $S_{m}$. It turns out that such a minimizer is actually given by $\hat s_{m}$. Since $|\M|=2^{N-1}$ is large for large values of $N$, we shall not consider the whole family of estimators $\ac{\hat s_{m},\ m\in\M}$ over which our selection procedure could be practically useless and rather focus on the (random) subfamily defined as follows. Let $\Lambda=\ac{1,\ldots,N}$ and define $\hat m(1)$ the partition of $[0,1)$ reduced to $\ac{[0,1)}$. Then for $\lambda\ge 2$, define by induction $\hat m(\lambda)$ as the random partition minimizing $\crit(N,\hat s_{m})$ among those $m\in\M$ satisfying both $\hat m(\lambda-1)\preceq m$ and $|m|=\lambda$ (in case of equality take one at random among the minimizers). Since for all $\lambda\in\Lambda$, $S_{\hat m(\lambda-1)}\subset S_{\hat m(\lambda)}$, note that the map $\lambda\mapsto \crit(N,\hat s_{\hat m(\lambda)})$ is decreasing with $\lambda$ and that $\hat m(N)$ corresponds to the regular partition based on the grid $\ac{k/N,\ k=0,\ldots,N}$. Finally, set for $\lambda\in\Lambda$, $\hat s_{\lambda}=\hat s_{\hat m(\lambda)}$. For such a family,  our procedure requires at most $N^{2}$ steps to obtain the family of partitions (for each value $\lambda$, finding $\hat m(\lambda)$ requires at most $N$ computations) and at most $N^{2}$ additional steps are required to proceed at the comparison pair by pair of the estimators $\hat s_{\lambda}$ to finally get $\tilde s=\hat s_{\hat \lambda}$. Consequently, the whole procedure requires of order $N^{2}$ steps and it follows from Corollary~\ref{BSA} that $\tilde s$ satisfies 
\[
C\E\cro{H^{2}\pa{s,\tilde s}}\le \inf_{\lambda\in\ac{1,\ldots,N}}\ac{\E\cro{H^{2}\pa{s,\hat s_{\lambda}}}+\tau\lambda\log(N-1)}.
\]

\section{Selecting among points}\label{sect-discret}
We assume here that the estimators $\hat s_{\lambda}$ are deterministic. In order to emphasize the fact that they do not depend on $N$, these will be denoted $s_{\lambda}$ hereafter. The aim of this section is to show that our selection procedure allows to select among arbitrary points in $\LL_{0}$ and also provides an alternative to the procedure based on testing proposed in Birg\'e~\citeyearpar{MR2219712} for the construction of $T$-estimators. The proofs of the following Propositions are delayed to Section~\ref{preuves-cor}. 

\subsection{Aggregation of arbitrary points}
Let $\ac{s_{\lambda},\ \lambda\in\Lambda}$ be a countable family of arbitrary points of $\LL_{0}$. Typically, one should think of the $s_{\lambda}$ as estimators of $s$ based on an independent copy $N'$ of $N$. In this case, with no loss of generality we may assume that $\Lambda=\M$ and $S_{m}=\ac{s_{m}}$ for all $m\in\M$. Then, the following result should be understood as conditional to $N'$.\\

\begin{prop}\label{agreg}
Assume that Assumption~\ref{bern} holds, set $\tau=4(2+cz)/(az^{2})$, and take $\ac{\Delta_{m},\ m\in\M}$ satisfying~\eref{sigma}. Choose  
$\tilde s=s_{\hat m}$ among $\ac{s_{m},\ m\in\M}$ according to Procedure~\ref{P1} with  
\[
\pen(s_{m})= z\tau\Delta_{m},\ \ \forall m\in\M.
\]
Then,  
\[
\E\cro{H^{2}\pa{s,\tilde s}}\le C\inf_{m\in\M}\cro{H^{2}\pa{s, s_{m}}+\tau\Delta_{m}}
\]
where $C$ depends on $z,b,M$ and $\Sigma$ only.
\end{prop}

Our procedure also allows to handle the problem of convex aggregation from i.i.d. observations in the same way as Birg\'e did in Section~9 of Birg\'e~\citeyearpar{MR2219712}. We shall not detail this in the present paper and rather refer to the paper by Birg\'e for examples and references.

\subsection{Selecting among discretized subsets of $\LL_{0}$}
For each $m\in\M$, let $\S_{m}=\ac{s_{\lambda},\ \lambda\in\Lambda(m)}$ be a discrete subset of $\LL_{0}$. Taking $\Lambda=\bigcup_{m\in\M}\Lambda(m)$, we consider the family $\ac{s_{\lambda},\ \lambda\in\Lambda}$ obtained by gathering all these discretization points.  The following holds. 

\begin{prop}\label{Lucien}
Let $M$ be a positive number and $\ac{\Delta_{m},\ m\in\M}$ a family of nonnegative numbers satisfying~\eref{sigma}. Assume that Assumptions~\ref{bern} and~\ref{h-discret} hold with $\tau=4(2+cz)/(az^{2})$. By applying Procedure~\ref{P1} with the family of estimators $\ac{s_{\lambda},\ \lambda\in\Lambda}$ and 
\[
\pen(t)= z\tau\inf\ac{4\eta_{m}^{2}+\Delta_{m},\ m\in\M(t)},\ \ \forall t\in\S
\]
the estimator $\tilde s$ satisfies 
\begin{equation}\label{pasmal}
\E\cro{H^{2}\pa{s,\tilde s}}\le C\inf_{m\in\M}\cro{H^{2}\pa{s,\S_{m}}+\tau\pa{\eta_{m}^{2}\vee \Delta_{m}}},
\end{equation}
where $C$ depends on $z,b,M$ and $\Sigma$ only.

If moreover $\S_{m}$ is a $(M,\eta_{m}\sqrt{\tau},D_{m})$-net for $S_{m}$ with $\eta_{m}^{2}=2(D_{m}\vee 1/8)$ for all $m\in\M$, then, 
\begin{equation}\label{b1}
\E\cro{H^{2}\pa{s,\tilde s}}\le C'\inf_{m\in\M}\cro{H^{2}\pa{s,S_{m}}+\tau\pa{D_{m}\vee \Delta_{m}}},
\end{equation}
where $C'$ depends on $z,b,M$ and $\Sigma$ only. 
\end{prop}

In density estimation, an inequality such as~\eref{b1} also holds for $T$-estimators as proven in Birg\'e~\citeyearpar{MR2219712} (see his Theorem~5). For suitable choices of collections $\ac{S_{m},\ m\in\M}$, an estimator $\tilde s$ satisfying~\eref{b1} possesses nice optimal properties (in the minimax sense) and outperforms in some situations the classical maximum likelihood estimator. For more details, we refer the reader to the paper of Birg\'e mentioned above. 

Assume now that for all (deterministic) $m\in\M$, one is able to build an estimator $\hat s_{m}$ (depending on $N$) with values in $S_{m}$ with a risk satisfying for some universal constant $C$,
\begin{equation}\label{eva}
\E\cro{H^{2}\pa{s,\hat s_{m}}}\le C\pa{H^{2}(s,S_{m})+D_{m}},\ \ \forall s\in\LL_{0}.
\end{equation}
By selecting among the family $\ac{\hat s_{m},\ m\in\M}$ with Procedure~\ref{P2}, one obtains an estimator $\tilde s'=\hat s_{\hat m}$ which also satisfies an inequality such as~\eref{b1} (this easily derives from~\eref{eqz}). Consequently, from a theoretical point of view both estimators $\tilde s$ and $\tilde s'$ possess similar properties. If the estimators $\hat s_{m}$ can be built in a simple way, the advantage of $\tilde s'$ compared to $\tilde s$ is rather practical since the former requires  the comparison pair by pair of the estimators $\hat s_{m}$ only although the latter requires that of all  the pairs of $s_{\lambda}$. This shows the use of the discretization device is actually useful only when no estimator $\hat s_{m}$ satisfying~\eref{eva} is available. This seems to be often the case when the models $S_{m}$ are not linear spaces or when the maximum likelihood estimator performs poorly.

Finally, we mention that a careful look at the proof of Theorem~5 in Birg\'e~\citeyearpar{MR2219712} shows that the selection rule described there could also be used to select among the estimators $\hat s_{m}$ in the sense that the resulting estimator would also satisfy an analogue of~\eref{b1}.

\section{Estimating the means of nonnegative random variables}\label{sect-GLM}
In this section, we consider the statistical setting described in Example~\ref{ex-means}.  Hereafter, we shall assume that $\sqrt{s}$ belongs to some closed convex subset $\overline \CC$ of $\R_{+}^{n}$. Since the distance $H$ between two elements $t,t'\in \R_{+}^{n}$ corresponds to the Euclidean distance between $\sqrt{t}$ and $\sqrt{t'}$, it seems natural to approximate the parameter $\sqrt{s}$ with respect to the Euclidean norm. To do so, we introduce a family of linear subspaces  $\ac{\overline V_{m},\ m\in\M}$ of $\R^{n}$ with respective dimensions denoted $\overline D_{m}$ that correspond to approximation spaces for $\sqrt{s}$. We associate to each of these the sets $V_{m}$ for $m\in\M$ which are either given by $V_{m}=\overline V_{m}\cap \overline \CC$ or  $V_{m}=\Pi_{\overline \CC}\overline V_{m}$. Finally, we consider the models $S_{m}$ defined for $m\in\M$ by
\[
S_{m}=\phi^{-1}\pa{V_{m}}=\ac{(u_{1}^{2},\ldots,u_{n}^{2}), \ u\in V_{m}}
\]
where $\phi(t)=\sqrt{t}$ for $t\in\R_{+}^{n}$. 

Two examples of collections $\ac{V_{m},\ m\in\M}$ are given below.

\begin{pb}[The regression problem]\label{pb10} 
Assume that $\sqrt{s}=\pa{F(x_{1}),\ldots,F(x_{n})}$ where the $x_{i}$ are deterministic points on $[0,1]$  and $F$ is a  function from $[0,1]$ into $\R_{+}$. Note that the problem we deal with can be written in a regression setting as follows 
\[
X_{i}=F^{2}(x_{i})+\varepsilon_{i},\ \ \ i=1,\ldots,n
\]
where the $\eps_{i}=X_{i}-F^{2}(x_{i})$ are independent and centered random variables. The problem is to estimate $s=(F^{2}(x_{1}),\ldots,F^{2}(x_{n}))$.
\end{pb}

In order to approximate $\sqrt{s}=(F(x_{1}),\ldots,F(x_{n}))$, it is natural to introduce linear spaces $\ac{\VV_{m},\ m\in\M}$ having good approximation properties with respect to usual classes of functions $F$ such as Besov spaces. For $\alpha>0$ and $p\in[1,+\infty]$, $B^{\alpha}_{p,\infty}(R)$ denotes the ball of radius $R>0$ of the Besov space $B^{\alpha}_{p,\infty}$. For a precise definition of these spaces, we refer to DeVore and Lorentz~\citeyearpar{DeVore}. The following result derives from Theorem~1 and Proposition~1 in Birg\'e \& Massart~\citeyearpar{MR1848840}.

\begin{prop}\label{app}
For all $r\in\N\setminus\ac{0}$ and $J\in\N$, there exists a family $\ac{\VV_{(\m,r)},\ \m\in\M_{r}(J)}$ of linear subspaces of $\IL^{2}([0,1],dx)$ and positive numbers $C(r),C'(r),C''(r)$ such that $D_{(\m,r)}=\dim(\VV_{(\m,r)})\le C(r)2^{J}$, $\log\ab{\M_{r}(J)}\le C'(r)2^{J}$ and for all $\alpha\in(1/p,r)$ and all $f\in\B^{\alpha}_{p,\infty}(R)$, 
\[
\inf\ac{\sup_{x\in [0,1]}\ab{f(x)-g(x)},\ g\in\bigcup_{(\m,r)\in\M_{r}(J)}\VV_{(\m,r)}}\le C''(r)R2^{-J\alpha}.
\]
\end{prop}
Thus, for handling Problem~\ref{pb10} we shall consider $\M=\bigcup_{r\ge 1}\bigcup_{J\ge 0}\M_{r}(J)$, and for all $m=(\m,r)\in\M$, take 
\[
\overline V_{m}=\ac{(g(x_{1}),\ldots,g(x_{n})),\ g\in\VV_{m}},\ V_{m}=\Pi_{\overline C}\overline V_{m}\ \ \text{and}\ \ \ S_{m}=\phi^{-1}(V_{m}).
\]
Besides, by taking for $m=(\m,r)\in\M_{r}(J)$, $\Delta_{m}=(C'(r)+1)2^{J}+r$ note that so that~\eref{sigma} holds since 
\[
\sum_{m\in\M}e^{-\Delta_{m}}\le \sum_{r\ge 1}\sum_{J\ge 0}\ab{\M_{r}(J)}e^{-(C'(r)+1)2^{J}-r}\le \sum_{r\ge 1}e^{-r}\sum_{J\ge 0}e^{-C'(r)2^{J}}<+\infty.
\]

Let us now turn to another problem.
\begin{pb}[The variable selection problem]\label{pb11} 
We assume that $\sqrt{s}$ is of the form
\[
\sqrt{s}=\sum_{j=1}^{p}\beta_{j}v^{(j)} 
\] 
where $\beta=(\beta_{1},\ldots,\beta_{p})$ is an unknown vector of $\R^{p}$ and  $v^{(1)},\ldots,v^{(p)}$ are $p\ge 2$ known vectors in $\R^{n}$. This means that the (squared) mean of each $X_{i}$ is a linear combination of the values $v_{i}^{(j)}$ of the predictor $v^{(j)}$ for $j=1,\ldots,p$ at experiment $i$. Since, the number of predictors $p$ may be large and possibly larger than the number $n$ of data, we shall assume that the vector $\beta$ is sparse which means that 
\[
\ab{\ac{j,\ \beta_{j}\neq 0}}\le \overline D_{\max}
\] 
for a known integer $\overline D_{\max}\le n$. Our aim is to estimate $\sqrt{s}$ and the set $\ac{j,\ \beta_{j}\neq 0}$.
\end{pb}

For this problem, we consider any class $\M$ of subsets  $m$ of $\ac{1,\ldots,p}$ with cardinality not larger than $\overline D_{\max}$, and define for $m\in\M$, $V_{m}=\overline V_{m}\cap\overline \CC$ where $\overline V_{m}$ is the linear span of the $v^{(j)}$ for $j\in m$ (with the convention $\overline V_{\varnothing}=\ac{0}$).

\subsection{Assumption on the $X_{i}$}
We assume  the following
\begin{ass}\label{H5}
The random variables $X_{i}$ are independent nonnegative random variable with respective means $s_{i}$  satisfying for some nonnegative numbers $\sigma$ and $\beta$ 
\begin{equation}\label{laplace}
\max_{i=1,\ldots,n}\E\cro{e^{u\pa{X_{i}-s_{i}}}}\le \exp\cro{{u^{2}\sigma s_{i}
\over 2(1-|u|\beta)}}\ \ \forall u\in (-1/\beta,1/\beta).
\end{equation}
\end{ass}

This assumption holds for a large class of distributions including, any random variables with values in $[0,\beta]$ (then $\sigma=\beta$), the Binomial distribution (then $\sigma=1=\beta$), the Poisson distribution (for the same choice of parameters), or the Gamma distribution $\gamma(p,q)$ (with mean $p/q$ and $\beta=1/q=\sigma$). By expanding~\eref{laplace} in a vicinity of 0, it is easy to see that Assumption~\ref{H5} implies that $\Var(X_{i})\le \sigma \E(X_{i})$ for all $i=1,\ldots,n$. 

In the remaining part of this section, under Assumption~\ref{H5}, we shall set 
\begin{equation}\label{def-taub}
\tau= {96(\sigma+\beta)\over z^{2}}.
\end{equation}

\subsection{Discretizing the $S_{m}$}\label{sect:discretisation}
To each $m\in\M$ such that $S_{m}\neq\ac{0}$, we apply  with $V=V_{m}$, $\overline V=\overline V_{m}$ and $S=S_{m}$ one of the two discretization procedures described below (accordingly to the form of $V_{m}$). These procedures lead to a discretized subset $\S_{m}$ of $S_{m}$ associated to a parameter $\eta=\eta_{m}$ depending on the dimension of $\overline V_{m}$. 

The first procedure below is abstract and is based on a discretization argument introduced in Birg\'e~\citeyearpar{MR2219712}. The resulting set $\S$, though difficult to build in practice, possesses nice properties with respect to the original set $S$. We shall not detail the construction of $\S$ here and rather refer the reader to the proof in Section~\ref{sect-Pdisc}. We only present its properties. We shall use them in order to obtain new results on the estimation of the parameter $s$. 

\paragraph{Discretization $P1$} We assume here that $S=\phi^{-1}(V)$ where $V$ is of the form $\Pi_{\overline \CC}\overline V$ for some  linear subspace $\overline V$ of $\R^{n}$ with dimension $\overline D\ge 1$. We associate to $S$ the parameter 
\begin{equation}\label{eta1}
\eta^{2}= 2\times 4.2\overline D
\end{equation}
together with a discretized subset $\S$ with the following properties.
\begin{prop}\label{descr1}
There exists a discretized subset $\S$ of $S$ which satisfies Assumption~\ref{h-discret} with $M=1$ and $\tau$ and $\eta$ given by~\eref{def-taub} and~\eref{eta1} respectively. Moreover, $H(t,\S)\le 4H(t,S)$ for all $t\in\overline \CC$.
\end{prop}

The procedure below is much simpler than the one above but unfortunately not as powerful. Yet, it turns to be enough to handle Problem~\ref{pb11}. 

\paragraph*{Discretization P2} We assume here that $S=\phi^{-1}(V)$ with $V$ is of the form $\overline V\cap\overline \CC$ where $\overline V$ is a linear subspace of $\R^{n}$ with dimension $\overline D\ge 1$. Let $\Pi_{V}$ be the projector onto the closed convex set $V$ and $\T$ the subset of $\overline V$ given by
\begin{equation}\label{def-tau}
\T=\ac{\frac{2\eta\sqrt{\tau}}{\sqrt{\overline D}}\sum_{j=1}^{\overline D}k_{j}u_{j},\; (k_{j})_{j=1,\ldots,\overline D}\in\Z^{\overline D}}
\end{equation}
where $\ac{u_{1},\ldots,u_{\overline D}}$ is an orthonormal basis of $\overline V$ and 
\begin{equation}\label{eta2}
\eta^{2}= 2\times 1.031\overline D.
\end{equation}
Keep only the elements of $\T$ which are at distance not larger than $\eta\sqrt{\tau}$ of $V$, that is, those of 
\[
\T(\eta)=\ac{t\in\T,\ \inf_{v\in V}\norm{t-v}\le \eta\sqrt{\tau}}
\]
and define finally
\[
\S=\phi^{-1}(\T')\ \ {\rm \it where }\ \ \T'=\ac{\Pi_{V}t,\ t\in \T(\eta)}.
\]

The subset $\S\subset S$ satisfies the following.

\begin{prop}\label{build}
The subset $\S$ is an $(1,\eta\sqrt{\tau},1.031\overline D_{m})$-net of $S$ with $\tau$ and $\eta$ given by~\eref{def-taub} and~\eref{eta2} respectively.
\end{prop}
The proof is delayed to Section~\ref{p-prop5}.

\subsection{The results}
We have at disposal the family of discretized subsets $\S_{m}$ of $S_{m}$ which have been built in the previous section. We recall that each of these $\S_{m}$ are associated to a parameter $\eta_{m}>0$.  We consider here the discretization points $\ac{s_{\lambda},\ \lambda\in\Lambda(m)}=\S_{m}$ for $m\in\M$  and the family of estimators $\ac{ s_{\lambda},\ \lambda\in\Lambda=\bigcup_{m\in\M}\Lambda(m)}$ obtained by gathering those.  For such a family, the following holds:

\begin{thm}\label{varselec}
Assume that Assumption~\ref{H5} holds and let $\left\{ \Delta_{m},\: m\in\M\right\} $ be a family of weights satisfying~\eref{sigma}. Choose $\tilde{s}=s_{\hat{\lambda}}$ among the family $\ac{s_{\lambda},\ \lambda\in\Lambda}$
according to Procedure~\ref{P1} with $\pen$ satisfying
\begin{equation}\label{pen10}
\pen(t)=   z\tau\inf_{m\in\M(t)}\pa{4\eta_{m}^{2}+\Delta_{m}}\ \ \forall t\in\S.
\end{equation}
Then, 
\[
C\E\cro{H^{2}\pa{s,\tilde{s}}}\le \inf_{m\in\M}\ac{H^{2}(s,S_{m})+\tau\pa{\overline D_{m}\vee \Delta_{m}}}
\]
where $\tau$ is defined by~\eref{def-taub} and $C$ depends on $z$ and $\Sigma$ only. 
\end{thm}

We deduce from Theorem~\ref{varselec} the following risk bounds:

\begin{cor}\label{minimax}
Assume Assumption~\ref{H5} holds. Then, 
\begin{itemize}
\item[$(i)$] for any $m\in\M$, there exists an estimator $\tilde s_{m}$  satisfying
\begin{equation}\label{minimax1}
\sup_{s\in S_{m}}\E\cro{H^{2}\pa{s,\tilde s_{m}}}\le C\pa{\overline D_{m}\vee 1},
\end{equation}
where $C$ depends on $z,\sigma$ and $\beta$ only;
\item[$(ii)$] for Problem~\ref{pb10}, there exists an estimator $\tilde F$ such that for all $p\in[1,+\infty]$, $\alpha>{1/p}$ and $R\ge 1/n$
\[
\sup_{F\in B^{\alpha}_{p,\infty}(R)}\E\cro{{1\over n}\sum_{i=1}^{n}\pa{F(x_{i})-\tilde F(x_{i})}^{2}}\le C R^{2/(1+2\alpha)}n^{-2\alpha/(1+2\alpha)},
\]
where $C$ depends on $R,\alpha,p,\sigma,z$ and $\beta$;
\item[$(iii)$] for Problem~\ref{pb11}, by applying Procedure~\ref{P1} with weights $\Delta_{m}$ satisfying~\eref{sigma}, one selects a family of predictors $\ac{v^{(j)},\ j\in\hat m}$ and builds an estimator $\tilde s\in S_{\hat m}$ such that, 
\[
\E\cro{H^{2}\pa{s,\tilde s}}\le C\inf_{m\in\M}\cro{H^{2}(s,S_{m})+|m|\vee \Delta_{m}},
\]
where $C$ depends on $z,\sigma,\beta$ only.
\end{itemize}
\end{cor}

To our knowledge, Example~\ref{ex-means} has received little attention in the literature, especially from a non-asymptotic point of view. The only exceptions we are aware of are Antoniadis, Besbeas and Sapatinas~\citeyearpar{MR1897045} (see also Antoniadis and Sapatinas ~\citeyearpar{MR1859411}) and Kolaczyk and Nowak~\citeyearpar{MR2060167}. These papers consider the case where $s$ is of the form $(F(x_{1}),\ldots,F(x_{n}))$ for some function $F$ on $[0,1]$. In Antoniadis, Besbeas and Sapatinas~\citeyearpar{MR1897045}, the authors estimate $F$ by a wavelet shrinkage procedure and show that the resulting estimator achieves the  usual estimation rate of convergence over Sobolev classes with smoothness indexes larger than $1/2$.  Kolaczyk and Nowak~\citeyearpar{MR2060167} study the risk properties of some thresholding and partitioning estimators. There approach requires that the $s_{i}$ be bounded from above and below by positive numbers. Finally, Baraud and Birg\'e~\citeyearpar{MR2449129} tackle this problem but their approach restricts to the case of histogram-type estimators. In particular, the estimation rates they get hold for $\alpha\le 1$ only.

\subsection{Lower bounds}
The aim of this section is to show that the upper bound~\eref{minimax1} gives the right order of magnitude for the minimax rate over $S_{m}$, at least under the following assumptions.

\begin{ass}\label{H10}
The distribution of the random vector $X=(X_{1},\ldots,X_{n})$ belongs to an exponential family of the form 
\begin{equation}\label{ptheta}
dP_{\theta}=\exp\cro{\sum_{i=1}^{n}\pa{\theta_{i}T(x_{i})-A(\theta_{i})}}\bigotimes_{i=1}^{n }d\nu(x_{i})\ \text{with}\ \ \theta\in\Theta^{n}
\end{equation}
where $\nu$ denotes some measure on $\R_{+}$,  $T$ is a map from $\R_{+}$ to $\R$, $\theta_{i}$ are parameters belonging to an open interval $\Theta$ such that 
\[
\Theta\subset\ac{a\in \R,\ \int \exp\cro{aT(x)}d\nu(x)<+\infty}
\]
and $A$ denotes a smooth function from $\Theta$ into $\R$ satisfying  $A''(a)\ne 0$ for all $a\in\Theta$.
\end{ass}
These families include Poisson, Binomial and Gamma distributions (among others). Besides, it is well known that $A$ is infinitely differentiable on $\Theta$ and under $P_{\theta}$, the $X_{i}$ satisfy 
\[
\E\cro{X_{i}}=A'(\theta_{i})=s_{i}\ \ {\rm and}\ \ \Var(X_{i})=A''(\theta_{i})> 0,\ \ \forall i=1,\ldots,n.\ \ 
\]
Therefore, the unknown parameter $\sqrt{s}$ necessarily belongs to the open cube $\CC=I^{n}$ where $I$ denotes the interval $\phi\pa{A'(\Theta)}$. 

We shall also assume that the parameter space $\Theta$ is such that the following holds. 
\begin{ass}\label{H11}
There exists some $\kappa>0$ such that for all $\theta\in\Theta^{n}$, under $P_{\theta}$
\begin{equation}\label{dom}
0< \E(X_{i})\le \kappa\Var\pa{X_{i}}\ \ \forall i=1,\ldots,n.
\end{equation}
\end{ass}  
Since $A',A''$ are continuous and positive functions, such an assumption is automatically fulfilled by choosing $\Theta$ such that $\overline{\Theta}$ is compact and $A'$ and $A''$ positive on $\overline{\Theta}$.

\begin{thm}\label{borneinf}
Let $\overline V$ be a linear subspace of $\R^{n}$ with dimension $\overline D\ge 1$ and $S=\phi^{-1}(\overline V\cap\CC)$. Define 
\[
\RR=\ac{r\in(0,(2\sqrt{\kappa})^{-1}),\ \ \exists u_{0}\in \overline V\cap \CC,\ \ac{u\in \overline V,\ \norm{u-u_{0}}\le r}\subset \CC}.
\]
Under Assumptions~\ref{H10} and~\ref{H11}, 
\[
\inf_{\hat s}\sup_{s\in S}\E_{s}\cro{H^{2}(s, \hat s)}\ge {\overline D\over 30}\sup_{r\in\RR}r^{2},
\]
with the convention $\sup_{\varnothing}=0$.
\end{thm}

\section{Estimation and variable selection in non-Gaussian regression}\label{sect-regression}
In this section, we use the notations of Example~\ref{ex-marginal} and assume that we observe the random variables $X_{1},\ldots,X_{n}$ satisfying
\[
X_{i}=f_{i}+\eps_{i},\ i=1,\ldots,n
\]
where $f=(f_{1},\ldots,f_{n})$ is an unknown vector of $\R^{n}$ and the $\eps_{i}$ i.i.d.  random variables with known density $q$ on $\R$. Hereafter, we consider a family of linear subspaces $\ac{\overline V_{m},\ m\in\M}$ of $\R^{n}$ with respective dimensions denoted $\overline D_{m}$ and $\ac{\hat f_{\lambda},\ \lambda\in\Lambda}$ a family of estimators of $f$ with values in $\bigcup_{m\in\M}\overline V_{m}$ based on the observation of $X=(X_{1},\ldots,X_{n})$. 

For example, when $f$ is assumed to be of the form $\pa{F(x_{1}),\ldots,F(x_{n})}$ for some function $F$ and points $x_{1},\ldots,x_{n}$ in $[0,1]$ one can use the collection of linear spaces introduced to takle Problem~\ref{pb10}. Alternatively, if one assumes that $f$ is of the form $f=\sum_{j=1}^{p}\beta_{j}v^{(j)}$ as in Problem~\ref{pb11}, one can use the collection of $\overline V_{m}$ defined there to perform variable selection. 

As possible estimators, one can associate to each $\overline V_{m}$ the least-squares estimator of $f$ in $\overline V_{m}$ defined as $\hat f_{m}=\Pi_{m}X$ where $\Pi_{m}$ is the orthogonal projector onto $\overline V_{m}$. It is well-known that 
\begin{equation}\label{riskmc}
\E\cro{\norm{f-\hat f_{m}}^{2}}=\norm{f-\Pi_{m} f}^{2}+\overline D_{m}\sigma^{2},
\end{equation}
where $\sigma^{2}$ denotes the variance of the $\eps_{1}$ (provided that it is finite). In the context of variable selection, many efforts have been done to design (practical) selection rules among the predictors. Among the most popular procedures, we mention the Lasso and the Dantzig selector described respectively in Tibshirani~\citeyearpar{MR1379242} and Cand\`es and Tao~\citeyearpar{MR2382644}. Given a family $\Lambda$ of such procedures, an alternative family of estimators for $f$ could be that given by $\ac{\hat f_{\hat m(\lambda)},\ \lambda\in\Lambda}$ where $\hat m(\lambda)$ corresponds to the family of predictors selected by the procedure $\lambda$ in $\Lambda$. 

We shall assume the following

\begin{ass}\label{reg}
There exists some known positive numbers $R$, $\underline R, \overline R$  such that $\max_{i=1,\ldots,n}\ab{f_{i}}\le R$ and for all $r,r'\in[-R,R]$, 
\begin{equation}\label{condense}
\underline R \ab{r-r'}\le h\pa{q_{r},q_{r'}}\le \overline R\ab{r-r'}
\end{equation}
where $q_{r}(x)=q(x-r)$ for all $x,r\in\R$ and $h$ is the Hellinger distance between the densities $q_{r}$ and $q_{r'}$.  
\end{ass}

Throughout, we denote by $\overline \CC$ the cube $[-R,R]^{n}$, $q_{g}=\pa{q_{g_{1}},\ldots,q_{g_{n}}}$ for $g\in\R^{n}$ and $\LL_{0}=\ac{q_{g},\ g\in\overline \CC}$. Assumption~\ref{reg} implies that $(\LL_{0},H)$ is almost isometric to $\pa{\overline \CC,\|\ \|}$.  

Assumption~\ref{reg} holds if $\sqrt{q}$ is regular enough (see Theorem~3A page 183 in Borovkov~\citeyearpar{MR1712750}). The quantities $\overline R$ and $\underline R$ then depend on the Fisher information. Let us now turn to some examples. 

If for some known $\theta>0$
\[
q(x)={\theta\over 2}e^{-\theta\ab{x}},\ \ x\in\R
\]
then, $h^{2}\pa{q_{r},q_{r'}}=1-e^{-\theta\ab{r-r'}/2}\pa{1+ \theta\ab{r-r'}/2}$ and~\eref{condense} holds with $\overline R^{2}=1/2$ and $\underline R^{2}=(1-e^{-\theta R}(1+\theta R))/(\theta^{2}R^{2})$. Assumption~\ref{reg} can also be met even though the $\eps_{i}$ have no finite moments. For example, this is the case for 
\[
q(x)={1\over 2(1+|x|)^{2}},\ \ x\in\R.
\]
Indeed, 
\[
h^{2}(q_{r},q_{r'})=\psi(|r-r'|)\ \ \ {\rm with}\ \ \ \psi(x)=1-{2(1+x)\log(1+x)\over x(2+x)}
\] 
and since $\psi(x)/x^{2}$ is decreasing on $\R_{+}$ and tends to $1/2$ when $x$ tends to $0^{+}$, Inequality~\eref{condense} holds with $\overline R^{2}=1/2$ and $\underline R^{2}=\psi(2R)/(2R^{2})$.

\subsection{The procedure and the results}
Throughout this section, $\tau=50z^{-2}$. To each estimator $\hat f_{\lambda}$ with $\lambda\in\Lambda$, we associate the estimator of $q_{f}$ given by $\hat s_{\lambda}=q_{\hat f_{\lambda}}$. Our selection procedure is based on a suitable family of discretized subsets of $\LL_{0}$. Let us introduce two of these. 

\subsubsection*{Collection $(\C1)$}  For all $m\in\M$, let us set 
\[
S_{m}=\ac{q_{g},\ g\in \Pi_{\overline \CC}\overline V_{m}}\ \ \text{and}\ \ \S_{m}=\ac{q_{g},\ g\in \T'_{m}}
\]
where $\T'_{m}$ is the discretized set $\T'$ obtained by applying Discretization~$P1$ with $\overline V=\overline V_{m}$ and 
\[
\eta^{2}=\overline \eta_{m}^{2}=2 \times {16\overline D_{m}\over 3\underline R^{2}}.
\]

\subsubsection*{Collection $(\C2)$} For all $m\in\M$, let us set $\overline \CC_{m}=\overline\CC\cap \overline V_{m}$
\[
S_{m}=\ac{q_{g},\ g\in\overline \CC_{m}}\ \ \text{and}\ \ \S_{m}=\ac{q_{g},\ g\in \T'_{m}}
\]
$\T'_{m}$ is the discretized set $\T'$ obtained by Discretization~$P2$ with $\overline V=\overline V_{m}$ and 
\[
\eta^{2}=\overline \eta_{m}^{2}=2 \times {1.031\overline D_{m}\over \underline R^{2}}\ .
\]
We obtain the following result.

\begin{thm}\label{th-reg}
Let $\ac{\hat f_{\lambda},\ \lambda\in\Lambda}$ be an arbitrary (countable) family of estimators with values in $\bigcup_{m\in\M}\overline V_{m}$ and $\ac{\Delta_{m},\ m\in\M}$ a family of weights fulfilling~\eref{sigma}. Assume that Assumption~\ref{reg} holds.

By applying Procedure~\ref{P2} with the family of estimators $\ac{\hat s_{\lambda},\ \lambda\in\Lambda}$, the family $\ac{\S_{m},\ m\in\M}$ given by Collection $\C_{1}$, and
\[
\pen(t)=z\tau\inf\ac{4\times {2\overline D_{m}\over 3}+\Delta_{m},\ m\in\M\pa{t}}\ \ \forall t\in\S,
\]
one selects from the data some $\hat \lambda\in\Lambda$ for which the estimator $\tilde f=\hat f_{\hat \lambda}$ satisfies 
for some constant $C$ depending on $z,\overline R, \underline R$ and $\Sigma$
\[
C\E\cro{\norm{f-\tilde f}^{2}}\le \inf_{\lambda\in\Lambda}\ac{\E\cro{\norm{f-\hat f_{\lambda}}^{2}}+\E\cro{\overline D_{\hat m(\lambda)}\vee \Delta_{\hat m(\lambda)}}}.
\]

By applying Procedure~\ref{P2} with the family of estimators $\ac{\hat s_{\lambda},\ \lambda\in\Lambda}$, the family $\ac{\S_{m},\ m\in\M}$ given by Collection $\C_{2}$, and
\[
\pen(t)=z\tau\inf\ac{4\times 2.1\overline D_{m}+\Delta_{m},\ m\in\M\pa{t}}\ \ \forall t\in\S,
\]
one selects from the data some $\hat \lambda\in\Lambda$ for which the estimator $\tilde f=\hat f_{\hat \lambda}$ satisfies for some constant $C$ depending on $z,\overline R, \underline R$ and $\Sigma$
\begin{eqnarray*}
C\E\cro{\norm{f-\tilde f}^{2}}&\le&\inf_{\lambda\in\Lambda}\ac{\E\cro{\norm{f-\hat f_{\lambda}}^{2}+\E\cro{B(\hat f_{\lambda})}+\E\cro{\overline D_{\hat m(\lambda)}\vee \Delta_{\hat m(\lambda)}}}}
\end{eqnarray*}
where 
\[
B(\hat f_{\lambda})=\inf\ac{\norm{\hat f_{\lambda}-t}^{2},\ t\in\overline \CC_{\hat m(\lambda)}}.
\]
\end{thm}

If the family of estimators $\hat f_{\lambda}$ take their values in $\pa{\bigcup_{m\in\M}\overline V_{m}} \cap \overline \CC$ then $B(\hat f_{\lambda})=0$ and the same risk bound for $\tilde f$ is achievable with both Collections $\C2$ and $\C1$.

The proof of Theorem~\ref{th-reg} is postponed to Section~\ref{sect-th-reg}.

For illustration, we deduce the following corollaries in the context of variable selection. Hereafter, we consider the family of linear spaces $\ac{\overline V_{m},\ m\in\M}$ given in Problem~\ref{pb11}. 

\begin{cor}
For $m\in\M$, let $\ac{f_{\lambda},\ \lambda\in\Lambda(m)}$ be any countable and dense subset of $\overline V_{m}$. Define $\hat m(\lambda)=m$ if $\lambda\in \Lambda(m)$ and apply the procedure described in Theorem~\ref{th-reg} with the collection $\C_{1}$ and the family of estimators $\ac{f_{\lambda},\ \lambda\in\bigcup_{m\in\M}\Lambda(m)}$. Under Assumption~\ref{reg}, one selects a subset of predictors $\ac{v^{(j)},\ j\in\hat m(\hat \lambda)}$ for which the estimator  $\tilde f=\hat f_{\hat \lambda}\in\overline V_{\hat m(\hat \lambda)}$ satisfies
\[
\E\cro{\norm{f-\tilde f}^{2}}\le C\inf_{m\in\M}\ac{\norm{f-\Pi_{m}f}^{2}+\overline D_m\vee \Delta_{m}}
\]
where $C$ depends on $z,\overline R,\underline R$ and $\Sigma$.
\end{cor}
Provided that the distribution of the errors is known and the mean $f$ bounded by some known constant, this result shows that complete variable selection is possible even though the errors may not admit any finite moments. 

Let us now turn to some result showing how to select among families of least-squares estimators $\ac{\hat f_{m},\ m\in\M}$ as those introduced at the beginning of the section. Hereafter we take, $\Lambda=\M$, choose $\hat m(\lambda)=m$ for all $\lambda\in\Lambda$ and define 
$m^{*}$ as any minimizer of $|m|\vee \Delta_{m}$ among those $m\in\M$ such $f\in\overline \CC_{m}$.

\begin{cor}\label{derder}
Assume that $\sigma<+\infty$ and that Assumption~\ref{reg} holds true. Let  $\ac{\Delta_{m},\ m\in\M}$ be a family of weights satisfying~\eref{sigma}. Consider the family of least-squares estimators $\ac{\hat f_{m}=\Pi_{m}X,\ m\in\M}$ and apply the selection procedure described in Theorem~\ref{th-reg} with the collection $(\C2)$. The resulting estimator $\tilde f\in V_{\hat m}$ satisfies, 
\[
C\E\cro{\norm{f-\tilde f}^{2}}\le \E\cro{\norm{f-\hat f_{m^{*}}}^{2}}\vee\Delta_{m^{*}}
\]
where $C$ depends on $z,\overline R,\underline R,\Sigma$ and $\sigma$.
\end{cor}

\begin{proof}
Note that $B(\hat f_{m^{*}})\le \norm{\hat f_{m^{*}}-f}^{2}$ since $f\in\overline \CC_{m^{*}}$. The result follows by applying Theorem~\ref{th-reg} and choosing $\lambda=m^{*}$ to bound the infimum from above.
\end{proof}

\section{Proofs}\label{sect:proof}
\subsection{Proof of Theorem~\ref{main}}\label{sect-pmain}
Throughout, $\kappa=z+1/\sqrt{2}$. Hereafter, we fix some estimator $\hat s_{\lambda}$ and assume first that $\EE(\hat s_{\lambda})\neq \varnothing$. Therefore, there exists $\hat s_{\lambda'}\in\EE(\hat s_{\lambda})$ with $\hat s_{\lambda'}\neq \hat s_{\lambda}$. By using Proposition~\ref{var2} with $r=(\hat s_{\lambda}+\hat s_{\lambda'})/2$ and the fact that $\gT(N,\hat s_{\lambda},\hat s_{\lambda'})\ge 0$, we get
\begin{eqnarray*}
H^{2}(s,\hat s_{\lambda'})-H^{2}(s,\hat s_{\lambda})&=& \cro{\rho\pa{s,\hat s_{\lambda}}-{1\over 2}\int_{\X}\hat s_{\lambda}d\mu}-\cro{\rho\pa{s,\hat s_{\lambda'}}-{1\over 2}\int_{\X}\hat s_{\lambda'}d\mu}\\
&=& -\gT(N,\hat s_{\lambda},\hat s_{\lambda'})+\pen(\hat s_{\lambda})-\pen(\hat s_{\lambda'})\\
&&\ +  \cro{\rho\pa{s,\hat s_{\lambda}}-\rho_{r}\pa{sd\mu,\hat s_{\lambda}}}-\cro{\rho\pa{s,\hat s_{\lambda'}}-\rho_{r}\pa{sd\mu,\hat s_{\lambda'}}}\\
&&\ + \cro{\rho_{r}\pa{sd\mu,\hat s_{\lambda}}-\rho_{r}\pa{sd\mu,\hat s_{\lambda}}}-\cro{\rho_{r}\pa{sd\mu,\hat s_{\lambda'}}-\rho_{r}\pa{sd\mu,\hat s_{\lambda'}}}\\
&\le& {1\over \sqrt{2}}\cro{H^{2}\pa{s,\hat s_{\lambda}}+H^{2}\pa{s,\hat s_{\lambda'}}}\\
&&\ + Z(N,\hat s_{\lambda},\hat s_{\lambda'})+\pen(\hat s_{\lambda})-\pen(\hat s_{\lambda'})
\end{eqnarray*}
and therefore,
\[
\pa{1-{1\over \sqrt{2}}}H^{2}\pa{s,\hat s_{\lambda'}}
\le\pa{1+{1\over \sqrt{2}}}H^{2}\pa{s,\hat s_{\lambda}}+ Z(N,\hat s_{\lambda},\hat s_{\lambda'})+\pen(\hat s_{\lambda})-\pen(\hat s_{\lambda'}).
\]
For $\xi>0$, let us set 
\[
y^{2}(m,m',\xi)= \tau \pa{d_{m}\vee d_{m'}+\Delta_{m}+\Delta_{m'}+\xi},
\]
and
\[ 
\Omega_{\xi}=\bigcap_{(m,m')\in\M^{2}}\ac{\sup_{(t,t')\in S_{m}\times S_{m'}}{Z(N,t,t')\over w^{2}(t,t',y(m,m',\xi))}\le z}.
\]
Note that under Assumption~\ref{hypo}, $\P\pa{\Omega_{\xi}}\ge 1-\gamma\Sigma^{2} e^{-\xi}$. On $\Omega_{\xi}$, 
\begin{eqnarray*}
Z(N,\hat s_{\lambda},\hat s_{\lambda'})&\le& zH^{2}(s,\hat s_{\lambda})+zH^{2}(s,\hat s_{\lambda'})\\
&& \ \ + z\inf\ac{y^{2}(m,m',\xi),\ (m,m')\in{\M}(\hat s_{\lambda})\times {\M}(\hat s_{\lambda'})}\\
&\le& zH^{2}(s,\hat s_{\lambda})+zH^{2}(s,\hat s_{\lambda'})\\
&& \ \ +z\tau\inf_{(m,m')\in\M(\hat s_{\lambda})\times\M(\hat s_{\lambda'})}\pa{d_{m}+d_{m'}+\Delta_{m}+\Delta_{m'}+\xi}
\end{eqnarray*}
and since for all $\lambda\in\Lambda$,
\[
\pen(\hat s_{\lambda})\ge z\tau\inf_{m\in\M(\hat s_{\lambda})}\pa{d_{m}+\Delta_{m}},
\]
we have
\[
\pa{1-\kappa}H^{2}\pa{s,\hat s_{\lambda'}}\le \pa{1+\kappa}H^{2}\pa{s,\hat s_{\lambda}}+2\pen(\hat s_{\lambda})+z\tau\xi.
\]
Let us set $\alpha=(1+\kappa)/(1-\kappa)$. Since $\hat s_{\lambda'}$ is arbitrary among $\EE(\hat s_{\lambda})$, 
we deduce that on $\Omega_{\xi}$,
\begin{eqnarray*}
\D(\hat s_{\lambda})&=&\sup_{\hat s_{\lambda'}\in\EE(\hat s_{\lambda})}H^{2}\pa{\hat s_{\lambda},\hat s_{\lambda'}}\\
&\le&  \pa{1+\sqrt{\alpha}}H^{2}\pa{s,\hat s_{\lambda}}+\pa{1+{1\over\sqrt{\alpha}}}\sup_{\lambda'\in\EE(\hat s_{\lambda})}H^{2}\pa{s,\hat s_{\lambda'}}\\
&\le& \pa{1+\sqrt{\alpha}}^{2}H^{2}\pa{ s,\hat s_{\lambda}}+{2(1+\sqrt{\alpha})\over \sqrt{1-\kappa^{2}}}\pen(\hat s_{\lambda})+{z(1+\sqrt{\alpha})\over\sqrt{1-\kappa^{2}}}\tau\xi.
\end{eqnarray*}

Note that this bounds is obviously true if $\EE(\hat s_{\lambda})=0$ since then $\D(\hat s_{\lambda})=0$. Now by using that $\D(\hat s_{\hat \lambda})\le \D(\hat s_{\lambda})+\tau$, we obtain
\begin{eqnarray*}
H^{2}\pa{s,\hat s_{\hat \lambda}}&\le& \pa{2+\sqrt{\alpha}}H^{2}\pa{s,\hat s_{\lambda}}+\pa{1+{1\over 1+\sqrt{\alpha}}}H^{2}\pa{\hat s_{\lambda},\hat s_{\hat \lambda}}\\
&\le& \pa{2+\sqrt{\alpha}}H^{2}\pa{s,\hat s_{\lambda}}+\pa{1+{1\over 1+\sqrt{\alpha}}}\pa{\D(\hat s_{\hat \lambda})\vee\D(\hat s_{\lambda})}\\
&\le&  \pa{2+\sqrt{\alpha}}H^{2}\pa{s,\hat s_{\lambda}}+\pa{1+{1\over 1+\sqrt{\alpha}}}\D(\hat s_{\lambda})+\pa{1+{1\over 1+\sqrt{\alpha}}}\tau\\
&\le& \pa{2+\sqrt{\alpha}}\cro{\pa{2+\sqrt{\alpha}}H^{2}\pa{s,\hat s_{\lambda}}+{2\over \sqrt{1-\kappa^{2}}}\pen(\hat s_{\lambda})+{z\tau\over z(1+\sqrt{\alpha})}}\\
&&\ +\ {z\pa{2+\sqrt{\alpha}}\over \sqrt{1-\kappa^{2}}}\tau\xi\\
&\le& C_{1}(z)\cro{H^{2}\pa{s,\hat s_{\lambda}}+\pen(\hat s_{\lambda})}+C_{2}(z)\tau\xi
\end{eqnarray*}
with $\kappa=z+1/\sqrt{2}$, $\alpha=(1+\kappa)/(1-\kappa)$ and 
\begin{eqnarray}
C_{1}(z)&=&\pa{2+\sqrt{\alpha}}\max\ac{\pa{2+\sqrt{\alpha}},{2\over \sqrt{1-\kappa^{2}}}+{1\over z(1+\sqrt{\alpha})}}\label{C1}\\
C_{2}(z)&=&{z\pa{2+\sqrt{\alpha}}\over \sqrt{1-\kappa^{2}}}.\label{C2}
\end{eqnarray}
Finally, we conclude by using that $\P\pa{\Omega_{\xi}}\ge 1-\gamma\Sigma^{2} e^{-\xi}$ and the fact that $\hat s_{\lambda}$ is arbitrary.

\subsection{Proof of Lemma~\ref{lem1}}
Let $\xi>0$ and 
\[
y^{2}\ge \tau\cro{4\pa{\eta_{m}^{2}\vee \eta_{m'}^{2}}+\xi}\ge 4\tau \pa{\eta_{m}^{2}\vee \eta_{m'}^{2}}.
\]
We set $C_{0}=(\S_{m}\cap\B(s,y))\times (\S_{m'}\cap\B(s,y))$ and for $j\ge 1$, 
\[
C_{j}=\ac{(t,t')\in \S_{m}\times \S_{m'},\ \ 2^{j-1}y^{2}<H^{2}\pa{s,t}+H^{2}\pa{s,t'}\le 2^{j}y^{2}}.
\]
Note that for all $j\ge 0$, $C_{j}\subset \pa{\S_{m}\cap\B(s,2^{j/2}y)}\times\pa{\S_{m}\cap\B(s,2^{j/2}y)}$ and that for $(t,t')\in C_{j}$,  $w^{2}(t,t',y)=\pa{H^{2}(s,t)+H^{2}(s,t')}\vee y^{2}\ge (2^{j-1}\vee 1)y^{2}$. By using Assumptions~\ref{bern} and~\ref{h-discret}, we get
\begin{eqnarray*}
\lefteqn{\P\cro{\sup_{(t,t')\in \S_{m}\times \S_{m'}}{Z(N,t,t')\over w^{2}(t,t',y)}\ge z}}\\
&\le& \sum_{(t,t')\in C_{0}}\P\cro{Z(N,t,t')\ge zy^{2}}+\sum_{j\ge 1}\sum_{(t,t')\in C_{j}}\P\cro{Z(N,t,t')\ge z2^{j-1}y^{2}}\\
&\le& b\ab{\S_{m}\cap\B(s,y)}\ab{\S_{m'}\cap\B(s,y)}\exp\cro{-{az^{2}y^{4}\over y^{2}+czy^{2}}}\\
&&\ +b\sum_{j\ge 1}|\S_{m}\cap\B(s,2^{j/2}y)||\S_{m'}\cap\B(s,2^{j/2}y)|\exp\cro{-{az^{2}2^{2(j-1)}y^{4}\over 2^{j}y^{2}+cz2^{j-1}y^{2}}}\\
&\le& bM^{2}\exp\cro{\pa{{1\over \tau}-{az^{2}\over 1+cz}}y^{2}} +bM^{2}\sum_{j\ge 1}\exp\cro{\pa{{1\over \tau}-{az^{2}\over 2(2+cz)}}2^{j}y^{2}}\\
&\le& bM^{2}\sum_{j\ge 0}\exp\cro{{2^{j}y^{2}\over \tau}}
\end{eqnarray*}
recalling that $\tau=4(2+cz)/(az^{2})$. By using that  
\[
\tau^{-1}y^{2}\ge 4(\eta_{m}^{2}\vee \eta_{m'}^{2})+ \xi\ge 1+\xi
\]
and the inequality $2^{j}\ge j+1$ which holds for all $j\ge 0$, we finally obtain
\begin{eqnarray*}
\P\cro{\sup_{(t,t')\in \S_{m}\times \S_{m'}}{Z(N,t,t')\over w^{2}(t,t',y)}\ge z}&\le& bM^{2}\sum_{j\ge 0}\exp\cro{-(j+1)(1+\xi)}\\
&\le& bM^{2}e^{-\xi}.
\end{eqnarray*}

\subsection{The proof of Theorem~\ref{main2}}
The proof relies on Theorem~\ref{main} and Lemma~\ref{lem1}. 

By definition of $\tilde s$, for all $\lambda\in\Lambda$
\begin{eqnarray}
H^{2}\pa{s,\tilde s}&\le& 2H^{2}\pa{s,\tilde s_{\tilde \lambda}}+2H^{2}\pa{\tilde s_{\tilde \lambda},\hat s_{\hat \lambda}}\nonumber \\
&\le& 2H^{2}\pa{s,\tilde s_{\tilde \lambda}}+2H^{2}\pa{\tilde s_{\tilde \lambda},\hat s_{\lambda}}+2\tau.\label{bis1}
\end{eqnarray}
By Lemma~\ref{lem1}, the collection of models $\ac{\S_{m},\ m\in\M}$ satisfies Assumption~\ref{hypo}, we can therefore apply Theorem~\ref{main} with the family of estimators $\ac{\tilde s_{\lambda},\ \lambda\in\Lambda}$ and get that with probability at least $1-\gamma\Sigma^{2}e^{-\xi}$ (with $\gamma=bM^{2}$), 
\begin{eqnarray*}
H^{2}\pa{s,\tilde s_{\tilde \lambda}}&\le& C(z)\cro{H^{2}\pa{s,\tilde s_{\lambda}}+\pen(\tilde s_{\lambda})+\tau (\xi+1)}\\
&\le& C(z)\cro{2H^{2}\pa{s,\hat s_{\lambda}}+2H^{2}\pa{\hat s_{\lambda},\tilde s_{\lambda}}+\pen(\tilde s_{\lambda})+\tau (\xi+1)}
\end{eqnarray*}
which with~\eref{bis1} and the fact that $\pen(\tilde s_{\lambda})\ge \tau z$ ($\Delta_{m}\ge 1$ for all $m$) leads to
\begin{eqnarray*}
H^{2}\pa{s,\tilde s}&\le& C'(z)\cro{H^{2}\pa{s,\hat s_{\lambda}}+A(\hat s_{\lambda},\S)+\tau(\xi+1)}
\end{eqnarray*}
and conclude the proof of the first part of Theorem~\ref{main2}. The second part follows by integration with respect to $\xi$. 

\subsection{Proof of Proposition~\ref{Examples}}\label{preuves-ex}
\paragraph{Case of Examples~\ref{ex-densite} and~\ref{ex-marginal}}
It suffices to prove the result in the case of Example~\ref{ex-marginal}, the result for Example~\ref{ex-densite} being obtained similarly by changing $Z(N,t,t')$ into $Z(N,t,t')/n$.

Note that for all $t,t'\in\LL_{0}$, 
\[
Z(N,t,t')=\sum_{i=1}^{n}\pa{\psi(t_{i},t'_{i},X_{i})-\E\cro{\psi(t_{i},t'_{i},X_{i})}}
\]
is a sum of independent and centered random variables bounded by $\sqrt{2}$. Besides, by setting $r_{i}=(t_{i}+t'_{i})/2$ for $i=1,\ldots, n$ and  using that for all $x_{i}\in\X_{i}$, $(t(x_{i})\vee t'_{i}(x_{i})/r_{i}(x_{i})\le 2$ we have
\begin{eqnarray*}
4\E\cro{Z^{2}(t,t',N)}&\le& \sum_{i=1}^{n}\int_{\X_{i}} \pa{\sqrt{t_{i}}-\sqrt{t'_{i}}}^{2}{s_{i}\over r_{i}}d\mu_{i}\\
&=&\sum_{i=1}^{n}\int_{\X_{i}} \pa{\sqrt{t_{i}}-\sqrt{t'_{i}}}^{2}\pa{\sqrt{{s_{i}\over r_{i}}}-1+1}^{2}d\mu_{i}\\
&\le& 2\sum_{i=1}^{n}\int_{\X_{i}} \pa{\sqrt{t_{i}}-\sqrt{t'_{i}}}^{2}\pa{\sqrt{{s_{i}\over r_{i}}}-1}^{2}d\mu_{i}\\
&& \ \ +2\sum_{i=1}^{n}\int_{\X_{i}} \pa{\sqrt{t_{i}}-\sqrt{t'_{i}}}^{2}d\mu_{i}\\
&=& 2\sum_{i=1}^{n}\int_{\X_{i}} {\pa{\sqrt{t_{i}}-\sqrt{t'_{i}}}^{2}\over r_{i}}\pa{\sqrt{s_{i}}-\sqrt{r_{i}}}^{2}d\mu_{i}+4H^{2}\pa{t,t'}\\
&\le& 8\pa{H^{2}\pa{s,r}+H^{2}\pa{s,t}+H^{2}\pa{s,t'}}.
\end{eqnarray*}
Since the concavity of $u\mapsto \sqrt{u}$ implies
$2H^{2}\pa{s,r}\le H^{2}\pa{s,t}+H^{2}\pa{s,t'}$, 
we have obtain that $t,t'\in\B(s,y)$
\[
{\rm Var}\pa{Z(t,t',N)}\le \E\cro{Z^{2}(t,t',N)}\le 3\cro{H^{2}\pa{s,t}+H^{2}\pa{s,t'}}\le 6y^{2}.
\]
By applying Bernstein's inequality, we obtain that Assumption~\ref{bern} is fulfilled with $b=1$, $a=1/6$ and $c=\sqrt{2}/6$. 

\paragraph{Case of Example~\ref{ex-poisson}}
In this case, 
\[
Z(t,t',N)=\int_{\X}\psi\pa{t,t',x}\pa{dN(x)-s(x)d\mu}
\]
where $\psi$ is bounded with values in $[-1/\sqrt{2},1/\sqrt{2}]$ and, by arguing as in Section~\ref{preuves-ex}, it satisfies
\[
\int_{\X}\psi^{2}\pa{t,t',x}s(x)d\mu\le 3\pa{H^{2}(s,t)+H^{2}(s,t')}\le 6y^{2}
\]
for all $t,t'\in\B(s,y)$. By applying Proposition~7 in Reynaud-Bouret~\citeyearpar{MR1981635} we obtain that $Z(N,t,t')$ satisfies~\eref{bern} with $a=1/12$, $b=1$ and $c=\sqrt{2}/36$.

\subsection{Proof of Corollary~\ref{cor2ter}}
Since for all $m\in\M$, $\eta_{m}^{2}=2(D_{m}\vee 1/8)$, Assumption~\ref{h-discret} holds. We can therefore apply Theorem~\ref{main2}. Since for all $\lambda\in\Lambda$, $\S_{m}$ is a $\eta_{\hat m(\lambda)}\sqrt{\tau}$-net for $S_{m}$ and since equality holds in~\eref{def-pen2}, 
\begin{eqnarray*}
A(\hat s_{\lambda},\S)&\le& \tau\pa{(1+4z)\eta_{\hat m(\lambda)}^{2}+z\Delta_{\hat m(\lambda)}}\\
&\le& \tau\pa{2(1+4z)(D_{\hat m(\lambda)}\vee (1/8))+z\Delta_{\hat m(\lambda)}}\\
&\le& \tau\pa{2(1+4z)(D_{\hat m(\lambda)})+\pa{1/4+2z}\Delta_{\hat m(\lambda)}},\ \forall \lambda\in\Lambda,
\end{eqnarray*}
which leads to the result.

\subsection{Proofs of Propositions~\ref{agreg} and~\ref{Lucien}}\label{preuves-cor}
\paragraph{Proof of Proposition~\ref{agreg}}
Take for $m\in\M=\Lambda$, $\S_{m}=S_{m}=\ac{s_{m}}$ and note that Assumption~\ref{h-discret} holds with $M=1$ and $\eta_{m}=0$. By Lemma~\ref{lem1}, Assumption~\ref{h-discret} with $d_{m}=0$ and $\gamma=b$ and the result follows by applying Corollary~\ref{cor1}.

\paragraph{Proof of Proposition~\ref{Lucien}}
Inequality~\eref{b1} follows from~\eref{pasmal} and the fact that since $\S_{m}$ is a $(M,\eta_{m}\sqrt{\tau},D_{m})$-net for $S_{m}$
\[
H^{2}(s,\S_{m})\le 2H^{2}(s,S_{m})+2\tau\eta_{m}^{2}\le 2H^{2}(s,S_{m})+4\tau(D_{m}\vee \Delta_{m}).
\]
Since Assumption~\ref{hypo} holds from Lemma~\ref{lem1}, we obtain~\eref{pasmal} by applying Corollary~\ref{cor1} with $\hat s_{\lambda}=s_{\lambda}$, noting that 
\[
\inf_{\lambda\in\Lambda}\cro{H^{2}(s,s_{\lambda})+\pen(s_{\lambda})}\le \inf_{m\in\M}\cro{H^{2}(s,\S_{m})+zt\pa{4\eta_{m}^{2}+\Delta_{m}}}.
\]

\subsection{Proof of Theorem~\ref{histo}}
It suffices to prove that Assumption~\ref{hypo} holds with $d_{m}=\delta|m|$ and then to apply Corollary~\ref{cor1}. 

Let $\xi>0$ and $y$ such that 
\[
y^{2}\ge \tau \pa{d_{m}\vee d_{m'}+\xi}.
\]
For $m,m'\in\M$, $t\in S_{m}$ and $t'\in S_{m'}$, $t$ and $t'$ are constant on each element $I\in m\vee m'$ with value $t_{I},t'_{I}$ respectively and therefore so is $\psi(t,t',.)$: 
\[
\psi(t,t',x)=\psi(t_{I},t_{I}')={1\over \sqrt{2}}\cro{\ \sqrt{{1\over 1+t_I/t'_I}}-{\sqrt{1\over 1+t'_I/t_I}}}, \ \ \forall x\in I.
\]
Consequently, by using that $\ab{\psi(t_{I},t_{I}')}\le 1/\sqrt{2}$ for all $I$ and Cauchy-Schwarz inequality
\begin{eqnarray*}
Z(N,t,t')&=& \sum_{I\in m\vee m'}\psi(t_{I},t'_{I})\pa{N(I)-\E\cro{N(I)}}\\
&=& \sum_{I\in m\vee m'}\psi(t_{I},t'_{I})\pa{\sqrt{N(I)}-\sqrt{\E\cro{N(I)}}}\pa{\sqrt{N(I)}+\sqrt{\E\cro{N(I)}}}\\
&=& \sum_{I\in m\vee m'}\psi(t_{I},t'_{I})\pa{\sqrt{N(I)}-\sqrt{\E\cro{N(I)}}}^{2}\\
&&\ \ +2 \sum_{I\in m\vee m'}\psi(t_{I},t'_{I})\sqrt{\E(N(I))}\pa{\sqrt{N(I)}-\sqrt{\E\cro{N(I)}}}\\
&\le& {\X^{2}(m\vee m')\over \sqrt{2}}+2\cro{\sum_{I\in m\vee m'}\psi^{2}(t_{I},t'_{I})\E(N(I))}^{1/2}\X(m\vee m')\\
&=& {\X^{2}(m\vee m')\over \sqrt{2}}+2\cro{\int \psi^{2}(t,t',x)sd\mu}^{1/2}\X(m\vee m')
\end{eqnarray*}
By arguing as in Section~\ref{preuves-ex}, we have that 
\[
\int_{\X}\psi^{2}(t,t',x)sd\mu\le 3\pa{H^{2}(s,t)+H^{2}(s,t')}
\]
and thus, by using that $w^{2}(t,t',y)\ge y^{2}$ and $w^{2}(t,t',y)\ge \pa{H^{2}(s,t)+H^{2}(s,t')}^{1/2}y$, we derive
\begin{eqnarray*}
\sup_{(t,t')\in S_{m}\times S_{m'}}{Z(N,t,t')\over w^{2}(t,t',y)} &\le& {\X^{2}(m\vee m')\over \sqrt{2}y^{2}}+2\sqrt{3}{\X(m\vee m')\over y}\\
&\le& {2\sqrt{6}+1\over \sqrt{2}}\pa{{\X^{2}(m\vee m')\over y^{2}}\vee{ \X(m\vee m')\over y}}.
\end{eqnarray*}
Since $z\in(0,1)$,
\begin{eqnarray*}
\ac{\sup_{(t,t')\in S_{m}\times S_{m'}}{Z(N,t,t')\over w^{2}(t,t',y)}\ge z}&\subset &\ac{{\X^{2}(m\vee m')\over y^{2}}\vee{ \X(m\vee m')\over y}\ge {z\sqrt{2}\over2\sqrt{6}+1}}\\
&\subset&\ac{{\X^{2}(m\vee m')\over y^{2}}\ge {2z^{2}\over \pa{2\sqrt{6}+1}^{2}}}
\end{eqnarray*}
and therefore
\begin{eqnarray*}
\P\cro{\sup_{(t,t')\in S_{m}\times S_{m'}}{Z(N,t,t')\over w^{2}(t,t',y)}\ge z}&\le& \P\cro{\X^{2}(m\vee m')\ge {2z^{2}y^{2}\over \pa{2\sqrt{6}+1}^{2}}}.
\end{eqnarray*}
We conclude by using Assumption~\ref{H3} with the fact that under Assumption~\ref{H4},
\[
y^{2}\ge \tau \pa{d_{m}\vee d_{m'}+\xi}\ge 
{\pa{2\sqrt{6}+1}^{2}\over 2z^{2}}\times a\pa{|m\vee m'|+\xi}.
\]

\subsection{Proof of Proposition~\ref{descr1}}\label{sect-Pdisc}
Given an orthonormal basis $\ac{u_{j},\ j=1,\ldots,\overline D}$, consider the $\eta\sqrt{\tau}$-net of $\overline V$ given by~\eref{def-tau}. Then, use Propositions~9 and~12 in Birg\'e~\citeyearpar{MR2219712} (with $\overline \pi=\Pi_{\overline \CC}$, $(M',d)=(\R_{+}^{n},\|\ \|)$, $\M_{0}=\overline \CC$, $T=\T$ and $\lambda=1=\eps$) in order to build a subset $\T'$ of $\Pi_{\overline\CC }\T$ with the properties (8.15) and (8.16) given there. Finally, set $\S=\phi^{-1}\pa{\T'}$. The properties of $\S$ derives from those of $\T'$ given in this Proposition~12.

\subsection{Proof of Proposition~\ref{build}}\label{p-prop5}
In the sequel, $d(.,.)$ denotes the Euclidean distance. By using Proposition~9 in Birg\'e~\citeyearpar{MR2219712}, $\T$ is a $\eta$-net for $\overline V$ satisfying for all $s\in\R^{n}$ and $r\ge  2\eta$, 
\begin{equation}\label{b100}
\ab{\ac{t\in \T,\ d(s,t)\le r}}\le \exp\cro{0.458\overline D\pa{{r\over \eta}}^{2}}.
\end{equation}
Since $V\subset \overline V$, for all $v\in V$ there exists $t\in \T(\eta)$ such that $d(v,t)\le \eta$ and 
\[
d(v,\T')\le d(v,\Pi_{V}t)\le d(v,t)\le \eta
\]
and $\T'$ is therefore an $\eta$-net for $V$. 

Let $s\in \R_{+}^{n}$. Assume that $\ac{t'\in \T',\ d(s,t')\le r}\neq\varnothing$. There exists $t$ in $\T(\eta)$ such that $t'=\pi_{V}t$ and $d(\Pi_{V}t,s)\le r$. For such a $t$,  
\[
d(s,t)\le d(s,\Pi_{V}t)+d(\Pi_{V}t,t)\le r+\eta
\]
and therefore
\begin{eqnarray*}
\ab{\ac{t'\in \T',\ d(s,t')\le r}}&\le& \ab{\ac{t\in \T(\eta),\ d(s,t)\le r+\eta}}\\
&\le& \ab{\ac{t\in \T,\ d(s,t)\le r+\eta}}.
\end{eqnarray*}
Of course the above inequality also holds if $\ac{t'\in \T',\ d(s,t')\le r}=\varnothing$. By using~\eref{b100} and the fact that $r+\eta\le 1.5 r$, we get
\begin{equation}\label{d11}
\ab{\ac{t'\in \T',\ d(s,t')\le r}}\le \exp\cro{1.031\overline D\pa{{r\over \eta}}^{2}}
\end{equation}
and conclude since $\phi$ is is an isometry from $(\R_{+}^{n},H)$ into $(\R_{+}^{n},d)$.

\subsection{Proof of Theorem~\ref{varselec}}
The proof is based on Proposition~\ref{Lucien}. Let us first check that the assumptions of this proposition are fulfilled. We already know from Propositions~\ref{descr1} and~\ref{build} that Assumption~\ref{h-discret} holds. It remains to check Assumption~\ref{bern}.
Under Assumption~\ref{H5}, we have for all $u=\pa{u_{1},\ldots,u_{n}}\in\R^{n}$ such that $\sum_{i=1}^{n}u_{i}^{2}s_{i}\le v^{2}$ and $\max_{i=1}^{n}\ab{u_{i}}\le \gamma$, and  all $\lambda\in (-1/(\beta \gamma),1/(\beta \gamma))$, 
\begin{eqnarray*}
\E\cro{e^{\lambda\sum_{i=1}u_{i}\pa{X_{i}-s_{i}}}}&=& \prod_{i=1}^{n}\E\cro{e^{\lambda u_{i}\pa{X_{i}-s_{i}}}}\\
&\le& \prod_{i=1}^{n}\exp\cro{{\lambda^{2}\sigma u_{i}^{2}s_{i}
\over 2(1-|\lambda|\gamma\beta)}}\\
&\le& \exp\cro{{\lambda^{2}\sigma v^{2}
\over 2(1-|\lambda|\gamma\beta)}}
\end{eqnarray*}
In particular, for all $\lambda\in (0,1/(\beta \gamma))$,  
\begin{equation}\label{cond-minimal}
\E\cro{e^{\lambda\sum_{i=1}u_{i}\pa{X_{i}-s_{i}}}}\le \exp\cro{{\lambda^{2}\sigma v^{2}
\over 2(1-\lambda \gamma\beta)}}.
\end{equation}
Under~\eref{cond-minimal}, we derive from Bernstein's inequality (see Massart~\citeyearpar{MR2319879}, Corollary 2.10), 
\begin{equation}\label{bernfd}
\P\cro{\sum_{i=1}^{n}u_{i}\pa{X_{i}-s_{i}}\ge \xi}\le \exp\cro{-{\xi^{2}\over 2(\sigma v^{2}+\gamma\beta\xi)}}.
\end{equation}
For $t,t'\in\B(s,y)\subset \R_{+}^{n}$, let us now take $u=(\psi(t,t',1),\ldots,\psi(t,t',n))$ (where $\psi$ is defined  by~\eref{def-psi} on $\X=\ac{1,\ldots,n}$) and note that 
\begin{eqnarray*}
\sum_{i=1}^{n}\psi(t,t',i)\pa{X_{i}-s_{i}}&=&Z(N,t,t')\\
\max_{i=1,\ldots,n}\ab{\psi(t,t',i)}&\le& {1\over \sqrt{2}}=\gamma.
\end{eqnarray*}
Besides, by arguing as in Section~\ref{preuves-ex},
\begin{eqnarray*}
 \sum_{i=1}^{n}\psi^{2}(t,t',i)s_{i}&=& {1\over 4}\sum_{i=1}^{n}{\pa{\sqrt{t_{i}}-\sqrt{t_{i}'}}^{2}s_{i}\over (t_{i}+t_{i}')/2}\\
 &\le& 3H^{2}(s,t)+3H^{2}(s,t')\le 6y^{2}=v^{2}.
\end{eqnarray*}
Consequently, we deduce from~\eref{bernfd} that Assumption~\ref{bern} is satisfied with $a=1/(12\sigma)$, $b=1$ and $c=\beta\sqrt{2}/(24\sigma)$ (then $\tau\le 96z^{-2}(\sigma+\beta)$). By applying the Proposition~\ref{Lucien}, we obtain~\eref{pasmal} from which we deduce Theorem~\ref{varselec} since for the Discretizations $P1$ and $P2$, the $\S_{m}$ satisfy 
\[
H^{2}(s,\S_{m})\le 16H^{2}(s,S_{m})+2\tau \eta_{m}^{2},\ \forall m\in\M.
\]

\subsection{Proof of Corollary~\ref{minimax}} Result~$(iii)$ is direct from Theorem~\eref{varselec}. For $(i)$, take with $\M=\ac{m}$, $\Delta_{m}=1$ and $\S_{m}$ a discretization of $S_{m}$ obtained by $P1$ or $P2$. Then, the result follows by applying Theorem~\ref{varselec} denoting $\tilde s$ by $\tilde s_{m}$. For~$(iii)$,  consider  the collection of models $S_{m}$ described to handle Problem~
\ref{pb10} and discretized them by applying $P1$. Apply Theorem~\ref{varselec} and take $\tilde F$ as any element of $\VV_{\hat m}$ such that $\sqrt{\tilde s}=(\tilde F(x_{1}),\ldots,\tilde F(x_{n}))$. We obtain that for all $F\in B^{\alpha}_{p,\infty}$
\[
\E\cro{n^{-1}H^{2}\pa{s,\tilde{s}}}=\E\cro{{1\over n}\sum_{i=1}^{n}\pa{F(x_{i})-\tilde F(x_{i})}^{2}}\le C\inf_{J\ge 0}\ac{R^{2}2^{-2J\alpha}+{2^{J}\over n}}
\]
and the result follows by optimizing with respect to $J$.

\subsection{Proof of Theorem~\ref{borneinf}}
In the sequel, $\rho(P,Q)$ and $h(P,Q)$ denote the Hellinger affinity and the Hellinger distance between the probabilities $P,Q$. For $\theta\in\Theta^{n}$, $A'(\theta)$ corresponds to the vector $t=(A'(\theta_{1}),\ldots,A'(\theta_{n}))$. We start with the following lemma.

\begin{lemma}\label{lem2}
Assume that Assumptions~\ref{H10} and~\ref{H11} hold. For all $\theta,\theta'\in\Theta^{n}$, $t=A'(\theta)$ and $t'=A'(\theta')$, we  have, 
\[
h^{2}\pa{P_{\theta},P_{\theta'}}\le -\sum_{i=1}^{n}\log\rho\pa{P_{\theta_{i}},P_{\theta'_{i}}}\le 4\kappa H^{2}(t,t').
\]
\end{lemma}

\begin{proof}
Since 
\begin{eqnarray*}
h^{2}\pa{P_{\theta},P_{\theta'}}&=&1-\rho\pa{P_{\theta},P_{\theta'}}= 1-\exp\cro{\sum_{i=1}^{n}\log\rho\pa{P_{\theta_{i}},P_{\theta'_{i}}}}\\
&\le& -\sum_{i=1}^{n}\log\rho\pa{P_{\theta_{i}},P_{\theta'_{i}}},
\end{eqnarray*}
it suffices to show that 
\[
-\sum_{i=1}^{n}\log\rho\pa{P_{\theta_{i}},P_{\theta'_{i}}}\le 4\kappa H^{2}(t,t')=4\kappa\sum_{i=1}^{n}H^{2}(t_{i},t'_{i}).
\]
By summing over $i$, it is enough to show the inequality for $n=1
$, what we shall do. Let $\theta,\theta'$ in $\Theta$ such that $t=A'(\theta)$ and $t'=A'(\theta')$. With no loss of generality, we may assume that $\theta'<\theta$ and set $\delta=(\theta-\theta')/2$. The Hellinger affinity between $P_{\theta}$ and $P_{\theta'}$ is given by
\[
\rho(P_{\theta},P_{\theta'})=\exp\cro{-\pa{{A(\theta)+A(\theta')\over 2}-A\pa{{\theta+\theta'\over 2}}}}
\]
and therefore
\begin{eqnarray*}
-\log\rho\pa{P_{\theta},P_{\theta'}}&=& {A(\theta)+A(\theta')\over 2}-A\pa{{\theta+\theta'\over 2}}\\
&=& {1\over 2}\cro{A(\theta)+A(\theta-2\delta)-2A\pa{\theta-\delta}}\\
&=& {1\over 2}\int_{\theta-\delta}^{\theta}\pa{A'(u)-A'(u-\delta)}du\\
&=& {1\over 2}\int_{\theta-\delta}^{\theta}\cro{\int_{u-\delta}^{u}A''(v)dv}du.
\end{eqnarray*}
Since $t,t'\in\R_{+}\setminus\ac{0}$ and since under Assumption~\ref{H11}, $A',A''$ do not vanish on $[\theta',\theta]$, for all $u\in[\theta-\delta,\theta]$ and $v\in[u-\delta,u]$
\begin{eqnarray*}
A''(v)&=& {A''(v)\over 2\sqrt{A'(v)}}{A''(u)\over 2\sqrt{A'(u)}}{4\sqrt{A'(v)A'(u)}\over A''(u)}\\
&\le& {A''(v)\over 2\sqrt{A'(v)}}{A''(u)\over 2\sqrt{A'(u)}}{4A'(u)\over A''(u)}\\
&\le& 4\kappa {A''(v)\over 2\sqrt{A'(v)}}{A''(u)\over 2\sqrt{A'(u)}}.
\end{eqnarray*}
giving thus,
\begin{eqnarray*}
-\log\rho\pa{P_{\theta},P_{\theta'}}&\le&2\kappa \int_{\theta-\delta}^{\theta}\cro{\int_{u-\delta}^{u} {A''(v)\over 2\sqrt{A'(v)}}{A''(u)\over 2\sqrt{A'(u)}}dv}du\\
&\le& 2\kappa \int_{\theta'}^{\theta}\cro{\int_{\theta'}^{\theta} {A''(v)\over 2\sqrt{A'(v)}}{A''(u)\over 2\sqrt{A'(u)}}dv}du\\
&=& 2\kappa \pa{\int_{\theta'}^{\theta}{A''(v)\over 2\sqrt{A'(v)}}dv}^{2}\\
&=& 2\kappa\pa{\sqrt{A'(\theta)}-\sqrt{A'(\theta')}}^{2}\\
&=&2\kappa\pa{\sqrt{t}-\sqrt{t'}}^{2}
\end{eqnarray*}
\end{proof}

The proof of Theorem~\ref{borneinf} is based on Assouad's Lemma (see Assouad~\citeyearpar{MR777600}), more precisely on the version given by Theorem~2.10 in Tsybakov~\citeyearpar{MR2013911}. In the sequel, $u_{1},\ldots,u_{\overline D}$ denote an orthonormal basis of $\overline V$ and $d(\eps,\eps')$ the Hamming distance between two elements $\eps$ and $\eps'$ of  $\ac{0,1}^{\overline D}$, that is  
\[
d(\eps,\eps')=\sum_{j=1}^{\overline D}\1_{\eps_{j}\neq \eps_{j}'}.
\]
Let $r\in\RR$. There exists $t^{0}\in S$ such that the Euclidean ball (of $\overline V$) centered at $u_{0}=\sqrt{t^{0}}$ with radius $r$ is contained in $\CC$. Consequently, there exists $\beta_{1},\ldots,\beta_{\overline D}$ such that $\sqrt{t^{0}}=\sum_{j=1}^{\overline D}\beta_{j}u_{j}$ and that for all $\eps\in\ac{0,1}^{\overline D}$ one can find $t^{\eps}\in S$ such that  
\[
\sqrt{t^{\eps}}=\sum_{j=1}^{\overline D}\pa{\beta_{j}+r\eps_{j}}u_{j}.
\]
Note that the for all $\eps,\eps'\in\ac{0,1}^{\overline D}$, 
\[
2H^{2}(t^{\eps},t^{\eps'})=\norm{\sqrt{t^{\eps}}-\sqrt{t^{\eps'}}}^{2}=r^{2}d(\eps,\eps').
\]
Besides, 
\begin{eqnarray*}
\inf_{\hat s}\sup_{s\in S}\E_{s}\cro{H^{2}\pa{s,\hat s}}&\ge & \inf_{\hat s}\sup_{\eps\in \ac{0,1}^{\overline D}}\E_{t^{\eps}}\cro{H^{2}\pa{t^{\eps},\hat s}}\\
&\ge & \inf_{\hat \eps}\sup_{\eps\in \ac{0,1}^{\overline D}}\E_{t^{\eps}}\cro{H^{2}\pa{t^{\eps},t^{\hat \eps}}}\\
&=& {r^{2}\over 2} \inf_{\hat \eps}\sup_{\eps\in \ac{0,1}^{\overline D}}\E_{t^{\eps}}\cro{d^{2}\pa{\eps,\hat \eps}},
\end{eqnarray*}
where the two last infimum run among all estimators $\hat \eps$ based on the observations $(X_{1},\ldots,X_{n})$ with values in $\ac{0,1}^{\overline D}$. Theorem~2.10 in Tsybakov~\citeyearpar{MR2013911} asserts that 
\[
\inf_{\hat \eps}\sup_{\eps\in \ac{0,1}^{\overline D}}\E_{t^{\eps}}\cro{d^{2}\pa{\eps,\hat \eps}}\ge {\overline D\over 2}\pa{1-\sqrt{\alpha(2-\alpha)}}
\]
provided that for all $\eps,\eps'$ such that $d(\eps,\eps')=1$, 
\[
h^{2}\pa{P_{\theta^{\eps}},P_{\theta^{\eps'}}}\le \alpha<1,
\]
where $\theta^{\eps}$ and $\theta^{\eps'}$ corresponds to the parameters in $\Theta$ associated to $t^{\eps}$ ad $t^{\eps'}$ respectively. By taking $\alpha=1/2$ and using Lemma~\ref{lem2}, for all $\eps,\eps'$ such that $d(\eps,\eps')=1$
\[
h^{2}\pa{P_{\theta^{\eps}},P_{\theta^{\eps'}}}\le 4\kappa H^{2}(t^{\eps},t^{\eps'})\le 2\kappa r^{2}\le {1\over 2}=\alpha.
\]
Therefore, 
\[
\inf_{\hat s}\sup_{s\in S}\E_{s}\cro{H^{2}\pa{s,\hat s}}\ge {1-\sqrt{3}/2\over 4}\ \overline Dr^{2},
\]
which concludes the proof since $r$ is arbitrary in $\RR$. 

\subsection{Proof of Theorem~\ref{th-reg}}\label{sect-th-reg}
The proof is based on Theorem~\ref{main2}. Let us first check that the assumptions of this theorem hold. The marginal of $X$ being given by $s=q_{f}$, we already know from Proposition~\ref{Examples} that Assumption~\ref{bern} holds true for Example~\ref{ex-marginal} with $a=1/6$, $b=1$ and $c=\sqrt{2}/36$ (which leads to the value $\tau=50z^{-2}$). In order to check Assumption~\ref{h-discret}, we distinguish between Collections $(\C1)$ and $(\C2)$. 

\paragraph{Case of Collection $(\C1)$}
For any $m\in\M$,  by using Propositions~9 and~12 in Birg\'e~\citeyearpar{MR2219712} with $\overline \pi=\Pi_{\overline \CC}$, $(M',d)=(\R^{n},\|\ \|)$, $\M_{0}=\overline \CC$, $\lambda=1=\eps$ and $T=\T_{m}$ where $\T_{m}=\T$ is given by~\eref{def-tau} as a discretization of the linear space $\overline V_{m}$, we obtain from the Discretization $P1$ a discretized subset $\T_{m}'$ of $\Pi_{\overline \CC}\overline V_{m}$  satisfying the properties (8.15) and (8.16) given Birg\'e~\citeyearpar{MR2219712}, that is for all $g\in\R^{n}$ and $r\ge \overline \eta_{m}/2$
\[
\ab{\ac{t\in\T_{m}',\ \norm{g-t}\le r\sqrt{\tau}}}\le \exp\cro{4.2\overline D_{m}\pa{{r\over \overline\eta_{m}}}^{2}}
\]
and $d(g,\T_{m}')\le 4d(g,\T_{m})$ (where $d(.,.)$ denotes the Euclidean distance). Since  for all $g,g'\in\overline \CC$, $H^{2}(q_{g},q_{g'})\ge \underline R^{2}\norm{g-g'}^{2}$, for all $g\in\overline\CC$ and $x\ge 2\eta_{m}=\underline R\overline \eta_{m}/2$ 
\begin{eqnarray*}
\ab{\S_{m}\cap \B(q_{g},x\sqrt{\tau})}&\le& \ab{\ac{g\in \T_{m}',\ \|g-g'\|\le \underline R^{-1}x\sqrt{\tau}}}\\
&\le&\exp\cro{4.2\underline R^{-2}\overline D_{m}\pa{{x\over \overline \eta_{m}}}^{2}}\\
&\le& \exp\cro{{\overline D_{m}\over 3}\pa{{x\over \eta_{m}}}^{2}}.
\end{eqnarray*}
Consequently, since $\eta_{m}=2\overline D_{m}/3$ for all $m$, the family $\ac{S_{m},\ m\in\M}$ satisfies Assumption~\ref{h-discret} with $M=1$. 

\paragraph{Case of Collection $(C2)$}
By using~\eref{d11} and arguing as in the previous case, for all $g\in\R^{n}$ and $x\ge 2\eta_{m}=2\underline R \overline \eta_{m}$,
\begin{eqnarray*}
\ab{\S_{m}\cap \B(q_{g},x\sqrt{\tau})}&\le& \ab{\ac{g'\in \T_{m}',\ \norm{g-g'}\le \underline R^{-1}x\sqrt{\tau}}}\\
&\le&\exp\cro{1.031\underline R^{-2}\overline D_{m}\pa{{x\over \overline \eta_{m}}}^{2}}\\
&\le& \exp\cro{1.031\overline D_{m}\pa{{x\over \eta_{m}}}^{2}}\\
\end{eqnarray*}
and we deduce similarly that the family $\ac{S_{m},\ m\in\M}$ satisfies Assumption~\ref{h-discret} with $M=1$.

Let us now finish the proof of Theorem~\ref{th-reg}. Since in both cases, Assumption~\ref{h-discret} holds, we can apply Theorem~\ref{main2} and get that 
\[
\E\cro{H^{2}\pa{q_{f},q_{\tilde f}}}\le C\inf_{\lambda\in\Lambda}\ac{\E\cro{H^{2}\pa{q_{f},q_{\hat f_{\lambda}}}+A(q_{\hat f_{\lambda}},\S)}}
\]
for some $C$ depending on $\Sigma$, $z$ only. Under Assumption~\ref{reg}, we derive 
that 
\begin{eqnarray*}
\lefteqn{\E\cro{\norm{f-\tilde f}^{2}}}\\
&\le& C{\overline R^{2}\over \underline R^{2}}\inf_{\lambda\in\Lambda}\ac{\E\cro{\norm{f-\hat f_{\lambda}}^{2}+\inf_{m\in\M}\pa{\inf_{t\in\T_{m}'}\norm{\hat f_{\lambda}-t}^{2}+\tau\overline R^{-2}\pa{\overline D_{m}\vee \Delta_{m}}}}}
\end{eqnarray*}

In the case of Collection $\C_{1}$, we conclude by using that for $m=\hat m(\lambda)$
\begin{eqnarray*}
\lefteqn{\inf_{m\in\M}\pa{\inf_{t\in\T_{m}'}\norm{\hat f_{\lambda}-t}^{2}+\tau\overline R^{-2}\pa{\overline D_{m}\vee \Delta_{m}}}}\\
&\le& 2\norm{f-\hat f_{\lambda}}^{2}+2\inf_{t\in\T_{m}'}\norm{f-t}^{2}+\tau\overline R^{-2}\pa{\overline D_{m}\vee \Delta_{m}}\\
&\le& 2\norm{f-\hat f_{\lambda}}^{2}+32\inf_{t\in\T_{m}}\norm{f-t}^{2}+\tau\overline R^{-2}\pa{\overline D_{m}\vee \Delta_{m}}\\
&\le& 2\norm{f-\hat f_{\lambda}}^{2}+64\inf_{t\in\overline V_{m}}\norm{f-t}^{2}+64\tau\overline \eta_{m}^{2}+\tau\overline R^{-2}\pa{\overline D_{m}\vee \Delta_{m}}\\
&\le& 66 \norm{f-\hat f_{\lambda}}^{2}+\tau \pa{64\times 11\underline R^{-2}+\overline R^{-2}}\pa{\overline D_{m}\vee \Delta_{m}}.
\end{eqnarray*}

For collection $\C_{2}$, we conclude by using that $\T_{m}'$ is a $\overline \eta_{m}\sqrt{\tau}$-net for $\overline \CC_{m}$ and that for $m=\hat m(\lambda)$, 
\begin{eqnarray*}
\lefteqn{\inf_{m\in\M}\pa{\inf_{t\in\T_{m}'}\norm{\hat f_{\lambda}-t}^{2}+\tau\overline R^{-2}\pa{\overline D_{m}\vee \Delta_{m}}}}\\
&\le& 2\norm{\hat f_{\lambda}-\Pi_{\overline \CC_{m}}\hat f_{\lambda}}^{2}+2\tau \overline \eta_{m}^{2}+\tau\overline R^{-2}\pa{\overline D_{m}\vee \Delta_{m}}.
\end{eqnarray*}


\subsection*{Acknowledgement} The author is thankful to Lucien Birg\'e for his careful reading of the paper and his thoughtful comments.

\bibliographystyle{apalike}

\end{document}